\documentclass{amsart} 

\usepackage[T1]{fontenc}
\usepackage[utf8]{inputenc}

\usepackage[a4paper,top=2cm,bottom=2cm,left=3cm,right=3cm,marginparwidth=1.75cm]{geometry}

\usepackage{amsthm} 
\usepackage{amssymb}
\usepackage{amsmath}
\usepackage{graphicx}
\usepackage{mathrsfs}
\usepackage[colorlinks=true, allcolors=blue]{hyperref}
\usepackage{dsfont}
\usepackage{amsrefs}
\usepackage{algorithm, algpseudocode}
\usepackage{mathtools}
\mathtoolsset{showonlyrefs}

\usepackage{xcolor}
\newtheorem{theorem}{Theorem}[section]
\newtheorem{lemma}[theorem]{Lemma}
\newtheorem{corollary}[theorem]{Corollary}
\newtheorem{proposition}[theorem]{Proposition}
\newtheorem{assumption}[theorem]{Assumption}

\usepackage{subcaption}
\usepackage{caption}
\usepackage{float}

\usepackage{siunitx}
\usepackage{booktabs}  
\usepackage{tabularx} 
\usepackage{multirow}

\usepackage[shortlabels]{enumitem}
\newlist{lista}{enumerate}{1}
\setlist[lista]{label=\alph*., nosep,leftmargin=*,align=right}

\newlist{listi}{enumerate}{1}
\setlist[listi]{label={\upshape(\roman*\upshape)},leftmargin=*,align=right, widest=iii,nosep, format=\bf}

\theoremstyle{definition}
\newtheorem{definition}[theorem]{Definition}

\theoremstyle{remark}
\newtheorem{remark}[theorem]{Remark}
\newtheorem{example}[theorem]{Example}

\numberwithin{equation}{section}
\DeclareMathOperator*{\dist}{dist} 

\DeclareMathOperator*{\argminA}{arg\,min} 

\DeclareMathOperator*{\circum}{circ}

\def\re{\mathds R}

\DeclareMathOperator{\SOL}{SOL}
\DeclareMathOperator{\VIP}{VIP}
\DeclareMathOperator{\rank}{rank}

\def\lV{\left\lVert }
\def\rV{\right\rVert }
\def\lv{\left\lvert }
\def\rv{\right\rvert}

\DeclareMathOperator{\dom}{dom}

\newcommand{\scal}[2]{\left\langle #1, #2 \right\rangle}

\def\EGM{\texttt{EGM}}
\def\AlgOne{\texttt{CRM-VIP1}}
\def\AlgTwo{\texttt{CRM-VIP2}}
\def\AlgBIOne{\texttt{BI1}}
\def\AlgBITwo{\texttt{BI2}} 

\def\AlgMalAdap{\texttt{Mal-Adap}}
\def\GT{\texttt{GT}}

\title{On circumcentered direct methods for monotone variational inequality problems}

\author{Roger Behling}
\address{Departamento de Matemática, Universidade Federal de Santa Catarina, Blumenau-SC 89065-300, Brazil}
\email{rogerbehling@gmail.com}
\thanks{R.~Behling was partially supported by Conselho Nacional de Desenvolvimento Cient\'ifico e Tecnol\'ogico -- CNPq (Grants XXXX 407147/2023-3 and 403197/2025-2) and Funda\c{c}\~ao de Amparo \`a Pesquisa do Estado do Rio de Janeiro -- FAPERJ (Grant E-26/201.345/2021).}

\author{Yunier Bello-Cruz}
\address{Department of Mathematical Sciences, Northern Illinois University, DeKalb-IL 60115-2828, USA}
\email{yunierbello@niu.edu}
\thanks{Y.~Bello-Cruz was partially supported by the National Science Foundation (Grant DMS-2307328).}

\author{Alfredo N. Iusem}
\address{School of Applied Mathematics, Funda\c{c}\~ao Get\'ulio Vargas, Rio de Janeiro-RJ 22250-900, Brazil}
\email{alfredo.iusem@fgv.br}

\author{Di Liu}
\address{Instituto de Matem\'atica Pura e Aplicada, Rio de Janeiro-RJ 22460-320, Brazil}
\email{di.liu@impa.br}
\thanks{D.~Liu was partially supported by CNPq (Grant 153172/2024-0).}

\author{Luiz-Rafael Santos}
\address{Departamento de Matemática, Universidade Federal de Santa Catarina, Blumenau-SC 89065-300, Brazil}
\email{l.r.santos@ufsc.br}
\thanks{L.-R.~Santos was partially supported by CNPq (Grants 310571/2023-5, 407147/2023-3 and 403197/2025-2) and by the Fundação de Amparo à Pesquisa e Inovação do Estado de Santa Catarina -- FAPESC  (Edital 21/2024, Grant 2024TR002238).}



\date{}

\begin{document}


\begin{abstract}
Circumcentered techniques have been shown to significantly accelerate projection-based methods for convex feasibility problems. Motivated by this success,  we propose two direct methods with circumcenter acceleration for solving variational inequality problems involving two classes of operators: paramonotone and monotone. Both schemes rely on approximate projections onto separating halfspaces, thereby avoiding computationally expensive exact projections. We establish convergence results for both methods and conduct numerical experiments, demonstrating that the proposed algorithms outperform classical methods, such as the extragradient algorithm, by orders of magnitude in terms of computational time, particularly when the feasible set is a complex intersection of convex sets.
\end{abstract}

\maketitle

\section{Introduction}

\subsection{Variational inequalities}

Let $C\subset\re^n$ be a nonempty, closed, and convex set, and let $F:\re^n\to\re^n$ be a continuous operator. The \emph{variational inequality problem} $\VIP(F,C)$ consists of finding some $x^*\in C$ such that
\begin{equation}
    \label{eq.VIP}
    \langle F(x^*),x-x^* \rangle  \geq 0,~~\forall x\in C.
\end{equation}
The set of solutions of $\VIP(F,C)$ will be denoted by  $\SOL(F,C)$. In particular, if $F$ is the gradient operator $\nabla g$ of a convex and differentiable function $g:\re^n\to\re$, then $\VIP(F,C)$ reduces to the first-order optimality condition for minimizing $g$ over $C$. Here and henceforth, we denote by $\langle \cdot,\cdot\rangle$ and $\lV\cdot\rV$ the usual Euclidean inner product and norm, respectively.

The variational inequality problem was initially introduced and studied by \cite{hart-stamp}. Since then, this problem has attracted considerable attention, as it covers a wide range of problems arising in different fields. For instance, many models in  economics, structural analysis, optimization, engineering sciences, and operations research can be posed as a VIP; see \cites{Allen:1977,Bnouhachem:2005,Ceng:2008,Fukushima:1992,Harker:1990,hart-stamp,Kinderlehrer:2000,Tremolieres:2011,IUSEM} and the references therein.   An excellent survey of methods for VIPs can be found in \cite{facch-pang-2}.

\subsection{Gradient-Type methods for variational inequalities}
\label{subs}
In this work, we are interested in \emph{direct methods} for solving $\VIP(F,C)$.
The basic idea consists of extending the projected gradient method
for constrained optimization, \textit{i.e.}, for the problem of minimizing
$g(x)$ subject to $x\in C$. This problem is a particular case of,
$\VIP(F,C)$ with $F=\nabla g$. This procedure is given by the following iterative scheme:
\begin{equation}
    \label{eq.GT}
    \begin{cases}
        x^0\in C,\\
        x^{k+1}=P_C(x^k-\alpha_k \nabla g(x^k)),
    \end{cases}
\end{equation}
with $\alpha_k>0$ for all $k$. The coefficients $\alpha_k$ are
called stepsizes and $P_C:\re^n\to C$ is the orthogonal
projection onto $C$, \textit{i.e.}, $P_C(x)=\argminA_{y\in C}\lV x-y\rV$.

An immediate extension of the method \eqref{eq.GT} to $\VIP(F,C)$ is the iterative procedure given
by
\begin{equation}
    \label{}
    \begin{cases}
        x^0\in C, \\
        x^{k+1}=P_C(x^k-\alpha_k F(x^k)),
    \end{cases}
\end{equation}
which will be called Gradient-Type (\GT) method for variational inequalities. Convergence results for this method require some monotonicity
properties of $F$. We introduce below several possible options.

\begin{definition}
\label{def.monotonicity}
    Consider an operator $F: \re^n \rightarrow \re^n$ and a convex set $C \subset\re^n$. $F$ is said to be 
    \begin{listi}
        \item \emph{monotone} on $C$ if $\langle F(x) - F(y), x-y \rangle \geq 0$ for all $x,y \in C$;
        \item \emph{paramonotone} on $C$ if it is monotone on $C$, and whenever $\langle x-y, F(x)-F(y) \rangle = 0$ with $x,y \in C$, it holds that $F(x)=F(y)$;
        \item \emph{strictly monotone} on $C$ if $\langle x-y, F(x)-F(y) \rangle > 0$ for all $x,y \in C$, $x\ne y$;
        \item \emph{uniformly monotone} on $C$ if $\langle x-y,F(x)-F(y) \rangle\ge\varphi(\lV x-y \rV)$ for all $x,y \in C$, where $\varphi: \re_+\to\re$ is an increasing function with $\varphi(0) = 0$;
        \item \emph{strongly monotone} on $C$ if $\langle x-y,F(x)-F(y) 
        \rangle\ge\omega\lV x-y\rV^2$ for all $x,y \in C$ and some $\omega >0$.
    \end{listi}
\end{definition}

\begin{remark}
It follows from Definition \ref{def.monotonicity} that the
following implications hold: (iv) $\Rightarrow$
(iii) $\Rightarrow$ (ii) $\Rightarrow$ (i). The reverse assertions
are not true in general.
\end{remark}

We mention that all these monotonicity notions can also be defined also for point-to-set operators \cite{Combettes:2018}, \textit{i.e.}, for $F:\re^n\to{\mathcal P}(\re^n)$.
In this situation, $F$ is usually required to be maximal monotone \cite[Definition 20.20]{bauschkeConvexAnalysisMonotone2017}.
In the point-to-point setting (the one considered in this paper),
maximal monotonicity reduces to monotonicity and continuity. We make this observation because the proofs of some results that will be invoked later require maximal monotonicity of $F$. Our continuity and
monotonicity assumptions ensure that those results remain valid in our case.

It has been proved in \cite{fang} that when $F$ is strongly
monotone and Lipschitz continuous, \textit{i.e.}, there exists $L>0$ such
that $\lV F(x)-F(y)\rV\leq L\lV x-y\rV$ for all $x, y\in \re^n$, then
the GT method 
converges to the unique solution of $\VIP(F,C)$, provided that
$\alpha_k\in \left(\epsilon,\frac{2\omega}{L^2}\right)$ for all
$k$ and for some $\epsilon>0$, where $\omega>0$ is the strong monotonicity constant.

These results were later improved under weaker monotonicity 
assumptions on the operator $F$.
The case of uniformly monotone operators was
analyzed in \cite{alber} and the case of paramonotone operators in
\cite{Bello:2012}. The first algorithm proposed in this paper is shown to converge under the assumption that $F$ is paramonotone. Thus, it is worthwhile to look more carefully at this class of operators.

The notion of paramonotonicity was introduced
in \cite{Bru},
and many of its properties were
established in
\cite{Censor:1998} and
\cite{iusem-paramonotono}. Among them, we mention
the following:
\begin{listi}
\item If $F$ is the subdifferential of a convex function, then
$F$ is paramonotone, see Proposition $2.2$ in
\cite{iusem-paramonotono}. 
\item If
$F:\re^n\to\re^n$ is monotone and differentiable, and
$J_F(x)$ denotes the Jacobian matrix of $F$ at $x$, then $F$ is
paramonotone if and only if $\rank(J_F(x)+J_F(x)^\top )=\rank(J_F(x))$
 for all $x$; see Proposition $4.2$ in \cite{iusem-paramonotono}.
\end{listi}
 It follows that affine operators of the form $F(x)=Ax+b$ are
paramonotone when $A$ is positive semidefinite (not necessarily
symmetric), and $\rank(A+A^\top )=\rank(A)$. This situation
includes cases of nonsymmetric and singular matrices, where
the solution set of $\VIP(F,C)$ can be a subspace, in contrast to the case of strictly or strongly monotone operators, for which the solution set is always a singleton, when nonempty \cite{facch-pang-2}. Of course, this can also happen for nonlinear operators, \emph{e.g.}, when $F=\nabla g$ for a convex function $g$, in which case the solution set is the set of minimizers of $g$ on $C$, which need not be a singleton. On the other hand, when $A$ is nonzero and skew-symmetric, $F(x)=Ax+b$ is monotone but not paramonotone. 

\begin{remark}  
There is no general way to relax the assumption on $F$
to plain monotonicity. For example, consider $F:
\re^2\to\re^2$ defined as $F(x)=Ax$, with
$A=\left(%
\begin{array}{rr}
  0 & \,\,1 \,\, \\
  -1 & \,\,0\,\,  \\
\end{array}%
\right)$. 
$F$ is monotone, and the unique solution of $\VIP(F,C)$ is
$x^*=0$. However, it is easy to check that 
$\lV x^k-\alpha_k F(x^k)\rV>\lV x^k\rV$ for all $x^k\ne 0$ and all $\alpha_k>0$, and
therefore the sequence generated by the GT method moves away
from the solution, regardless of the choice of the stepsize
$\alpha_k$.
\end{remark}

In order to overcome this difficulty, the following iteration, called the extragradient method,
was introduced by G.M. Korpelevich in \cite{kor}:
\begin{equation}
    \label{eq.EG}
    \begin{cases}
        y^k=P_C\left(x^k-\beta F(x^k)\right),\\
        x^{k+1}=P_C\left(x^k-\beta F(y^k)\right),
    \end{cases}
\end{equation}
where $\beta>0$ is a fixed number.  Assuming that $F$ is monotone
and Lipschitz continuous with constant $L$, that
$\beta\in\left(\epsilon,\frac{1}{L}\right)$ for some $\epsilon>0$, and that $\VIP(F,C)$ has solutions,
Korpelevich showed that the sequence $\{x^k\}$ generated by \eqref{eq.EG} converges to a solution of $\VIP(F,C)$. Numerous variants and improvements of this method have been proposed in the literature; see, for instance, \cites{quExtraGradientAndersonaccelerated2025,Mainge:2008,Censor:2011,Thong:2018,Iusem-Svaiter}.



\subsection{The case of $C=\cap_{i=1}^mC_i$}
In principle, these methods are direct, in the sense that there is no subproblem to be solved at each iteration, which requires just the evaluation of $F$ and the computation of the projection $P_C$ at certain points. However, 
the computation of the projection onto a convex set, in general,
is not elementary (except for some very special convex sets, such as balls or halfspaces), and hence it is worthwhile to explore situations in which the projection step can be replaced by less
expensive procedures. A typical situation is when the convex set
$C$ is given as an intersection of simpler sets, \textit{i.e.},
$C=\cap_{i=1}^mC_i$, where each $C_i\subset\re^n$ is closed and convex. For instance, in the very frequent case when $C=\{x\in\re^n:g_i(x)\le 0\, (1\le i\le m)\}$, with convex 
$g_i:\re^n\to\re$, it is natural to take 
$C_i=\{x\in\re^n:g_i(x)\le 0\}$. Usually, $P_{C_i}$ will be much easier to compute than $P_C$. 

A relaxed projection algorithm proposed by He and Yang in \cite{He:2019} for solving $\VIP(F,C)$ is as follows:
\begin{equation}
    \label{eq.he}
    x^{k+1} = 
    \left(P_m^k\circ\cdots\circ P_2^k\circ 
    P_1^k\right)(x^k - \alpha_k F(x^k)),
\end{equation}
where the $P_i^k$'s are approximations of $P_{C_i}$.
In the case of exact projections(\textit{i.e.}, $P_i^k=P_{C_i}$) this is
a variant of the GT method where the projection onto $C$ is replaced by sequential projections onto the sets $C_1, \dots C_m$.
The convergence analysis in \cite{He:2019} requires strong monotonicity of $F$. 

Our algorithms also work with easily computable approximate projections onto each set $C_i$, but our approach differs from
the one in \cite{He:2019} in several directions:
\begin{listi}
\item We consider a GT method which converges for paramonotone operators (rather than strongly monotone ones), under a standard coerciveness assumption,
and another method which converges for plainly monotone operators.
\item We use simultaneous (rather than sequential)
projections onto separating sets (e.g. halfspaces) of the sets $C_i$.
\item Most importantly, we accelerate the methods through the introduction of a circumcentering procedure, which we explain in the next subsection.
\end{listi}

 \subsection{Circumcentered-Reflection methods}

The Convex Feasibility Problem (CFP) consists of finding a point in the intersection of $m$ convex subsets of $\re^n$. Two classical methods for CFP are the Sequential and the Simultaneous Projection Methods (SePM and SiPM, respectively), which originate in \cite{Kaz} and \cite{Cim}.
 
If $C=\cap_{i=1}^mC_i$, with $C_i\subset\re^n$ closed and convex ($1\le i\le m$) and $P_i:\re^n\to C_i$ is the orthogonal projection onto $C_i$, then the iteration of SePM is of the form
\begin{equation} \label{eq1}
  x^{k+1} = 
    \left(P_m\circ\cdots\circ P_2\circ 
    P_1\right)(x^k),
\end{equation}
and the iteration of SiPM is given by
\begin{equation} \label{eq2}
  x^{k+1} = \frac{1}{m}\sum_{i=1}^mP_i(x^k).
\end{equation}
 
The Circumcentered-Reflection Method (CRM), introduced in \cite{Behling:2018}, starts with the case of two convex sets (say $C=A\cap B$) and has the following iteration formula:

\begin{equation}
    \label{eq.CRM}
    x^{k+1}: = T_{A,B}(x^k) = \circum(x^k,R_A(x^k),R_B(R_A(x^k))),
\end{equation}
where $P_A$, $P_B$ are the projection operators onto $A,B$, $R_A = 2P_A - {\rm I}$ and $R_B = 2P_B - {\rm I}$ are the standard reflection operators. The point $\circum(x^k,R_A(x^k),R_B(R_A(x^k)))$ is the circumcenter of $x^k$, $R_A(x^k)$ and $R_B(R_A(x^k))$, \textit{i.e.}, a point equidistant from these three points and lying in the affine manifold spanned by them. 
One can regard $\circum(\cdot,\cdot,\cdot)$ as a mapping from $\re^n\times\re^n\times\re^n$ to $\re^n$. It has been proved in \cite{Behling:2018} that the sequence generated in this way converges to a point in $C$ if $C$ is nonempty and either $A$ or $B$ is an affine manifold.

However, CRM can be applied to the general CFP with $m$ sets, using Pierra's product space reformulation (see \cite{Pierra:1984}), which works as follows:

Define ${\textbf W}, {\textbf D}\subset \re^{nm}$ as
\begin{equation}
    \label{eq.W}
    {\textbf W}\coloneqq C_1\times \cdots \times C_m \subset \re^{nm},
\end{equation}
and
\begin{equation}
    \label{eq.D}
    {\textbf D}\coloneqq \{{\textbf y}\in \re^{nm}: \textbf{y}=(x,\ldots,x),x\in \re^n\}.
\end{equation}
It is immediate that if ${\textbf y}=(z,\dots ,z)\in\re^{nm}$ then
$$
(P_{\textbf D}\circ P_{\textbf W})({\textbf y})=
(\bar z, \dots ,\bar z),
$$
with $\bar z=\frac{1}{m}\sum_{i=1}^m P_i(z)$.
Noting that if we consider the SiPM iteration with $x^k=z$ we
get $x^{k+1}=\bar z$, we conclude that an iteration of SePM in the product space $\re^{nm}$ with the two convex sets ${\textbf W},
{\textbf D}$,
starting at a point in ${\textbf D}$, is equivalent to an iteration of SiPM in $\re^n$ with the sets $C_1, \dots ,C_m$. Note also that
${\textbf D}$ is an affine manifold (indeed, a linear subspace).
Hence, applying CRM with the sets 
${\textbf W}, {\textbf D}\subset\re^{nm}$ instead of
$A,B\subset\re^n$, we get a sequence ${\textbf z}^k=
(x^k,\dots ,x^k)\in\re^{nm}$, with $\{x^k\}\subset\re^{nm}$ 
converging to a point $x^*\in 
C=\cap_{i=1}^m C_i$, where the $C_i$'s are now arbitrary closed and convex subsets of $\re^n$. We define now the circumcentered-reflection operator in the product space as follows:
\begin{equation}
\label{equ10}
{\rm \mathbf{T}_{\textbf{W},\textbf{D}}}({\textbf z}) = 
\circum(\mathbf{R}_{\mathbf{D}}(\mathbf{R}_{\mathbf{W}}(\mathbf{z})),\mathbf{R}_{\mathbf{W}}(\mathbf{z}),\mathbf{z}),
\end{equation}
where 
\begin{equation}
\label{eq.Reflection_product_set}
\mathbf{R}_{\mathbf{W}}(\mathbf{z}) = \left(R_{C_1}(z),\ldots,R_{C_m}(z)\right),
\end{equation} 
and 
\begin{equation}
\label{eq.Reflection_diagonal_space}
\mathbf{R}_{\mathbf{D}}(\mathbf{z}) =
\newline
\left(\frac{2}{m}\left(\sum_{i=1}^mP_{C_i}(z)\right)-z,\ldots,\frac{2}{m}\left(\sum_{i=1}^mP_{C_i}(z)\right)-z\right),
\end{equation}
with ${\mathbf z}=(z,\dots ,z)\in{\mathbf D}$. In principle, the operator $\mathbf{T}_{\mathbf{W},\mathbf{D}}$ is defined
from $\mathds{R}^{nm}$ to $\mathds{R}^{mn}$. However, by Proposition 2.7(i) in \cite{Behling:better}, we know that if $\mathbf{z} \in \mathbf{D}$, then $\mathbf{T}_{\mathbf{W},\mathbf{D}}(\mathbf{z}) \in \mathbf{D}$. Therefore, we conclude that $\mathbf{T}_{\mathbf{W},\mathbf{D}}:\mathbf{D} \rightarrow \mathbf{D}.$

The CRM method has proved to be quite superior to classical projection methods for solving the CFP, achieving superlinear convergence under mild assumptions, while the classical methods are at most linearly convergent; see
\cites{Behling:better,Bauschke:2018a,Bauschke:2020,Bauschke:2021,Bauschke:2022,Behling:ScCRM,Behling:PACA,Behling:cCRM,Behling:cone,Dizon:2020,Lindstrom:2022,Ouyang:2023,Ouyang:2021}.

\subsection{Approximate projections}

CRM is superior to most known methods for the CFP, but still has a serious drawback: the computation of the projections onto the sets $C_i$.
A variant of CRM which uses approximate projections in the product space setting was introduced in \cite{Behling:2022a} and denoted as ACRM (Approximate CRM). We begin by defining separating operators for convex sets.

 \begin{definition}\label{separ}
     Given a closed convex set $K \subset\re^n$, a separating operator for $K$ is a point-to-set mapping: $S:\re^n\to {\mathcal P}(\re^n)$ satisfying:
     \begin{listi}
         \item  $S(x)$ is closed and convex for  all $x\in \re^n$,
         \item  $K \subset S(x)$ for all $x\in \mathds{R}^n$, 
         \item  If a sequence $\{z^k\} \subset \mathds{R}^n$ converges to $z^* \in\re^n$ and 
         $\lim_{k\to\infty}\dist (z^k,S(z^k))=0$, then 
         $z^*\in K$,
         \item  For any convex and compact set 
         $V\subset{\mathbb R}^n$ there exists $\tau\in (0,1]$ such that
         \begin{equation}\label{equ111}
         \dist\left(P_{S(x)\cap V}(x),K\right)\le
         \tau\,\dist(x,K),
         \end{equation}
     \end{listi}
where $\dist(\cdot ,\cdot)$ denotes the Euclidean distance between a point and a set in $\re^n$. 
\end{definition}

This notion of {\it separating operator} was introduced in \cite{Behling:2022a}, where only items (i)-(iii) in Definition \ref{separ} were demanded. In this paper, we additionally require item (iv). We commit a slight abuse of notation, keeping the same definition after adding (iv), which prevents $S(x)$ from being too close to $K$. As long as we determine a separating operator $S_i$ for $C_i$ that satisfies the assumptions in Definition \ref{separ}, the exact projection onto $S_i$ can be regarded as the approximate projection onto the original set $C_i$. Suppose that $S_i(x) \subset \mathds{R}^n$ is a separating operator for $C_i$ at point $x \in \mathds{R}^n$ for each $i \in \{1,\ldots,m\}$. In order to utilize CRM with approximate projection, we first define the separating operator in the product space as follows:
\begin{equation}
    \label{eq.seperatingSet}
    \mathbf{S}\coloneqq  S_1\times \cdots \times S_m \subset \mathds{R}^{nm}.
\end{equation}
Then, the {\it approximate circumcentered-reflection operator} $\mathbf{T}_{\mathbf{S},\mathbf{D}}:\mathds{R}^{nm} \rightarrow \mathds{R}^{nm}$ is defined as
\begin{equation}
    \label{eq.ACRM_operator} \mathbf{T}_{\mathbf{S},\mathbf{D}}(\mathbf{z}) 
    =\circum(\mathbf{R}_{\mathbf{D}}(\mathbf{R}_{\mathbf{S}}(\mathbf{z})),\mathbf{R}_{\mathbf{S}}(\mathbf{z}),\mathbf{z}),
\end{equation}
where $\circum(\cdot,\cdot,\cdot)$ is the circumcenter operator  defined in \eqref{eq.CRM}, $\mathbf{R}_{\mathbf{S}}$ is defined in the same way as in \eqref{eq.Reflection_product_set} but with respect to $\mathbf{S}$ defined in \eqref{eq.seperatingSet} and $\mathbf{R}_{\mathbf{D}}$ is defined in \eqref{eq.Reflection_diagonal_space}. Therefore, the sequence $\{\mathbf{z}^k\} \subset \mathds{R}^{nm}$ generated by ACRM can be written as
\begin{equation}
    \label{eq.ACRM_productspace}
    \mathbf{z}^{k+1} = \mathbf{T}_{\mathbf{S}_k,\mathbf{D}}(\mathbf{z}^k), 
\end{equation}
where 
$$\mathbf{S}_k =S_i^k \times \cdots \times S_m^k \subset \mathds{R}^{nm},$$
$S_i^k$ is short for $S_i(x^k)$ which is any separating operator for $C_i$, and $\mathbf{z}^k = (x^k,\ldots,x^k)$. In view of Proposition 2.7(i) in \cite{Behling:better}, we conclude that 
$\mathbf{T}_{\mathbf{S}_k,\mathbf{D}}: \mathbf{D}:\rightarrow \mathbf{D}$, that is $\{\mathbf{z}^k\}\subset\mathbf{D}$ for all $k \geq 1$. Now we go back to $\mathbf{\re}^n$, and define the approximate circumcentered-reflection operator $T_{\mathbf{S}_k,\mathbf{D}}:\mathds{R}^n \rightarrow \mathds{R}^n$  as follows
\begin{equation}
    \label{eq.ACRM}
    T_{\textbf{S}_k,\textbf{D}}(x)= \bar x, 
\end{equation}
where 
\begin{equation}
\mathbf{\bar x}=\mathbf{T}_{\mathbf{S}_k,\mathbf{D}}(\mathbf{x}),\qquad 
\mathbf{\bar x}=(\bar x,\dots ,\bar x)\qquad
\mathbf{x}=(x,\dots ,x),
\end{equation}
with $\mathbf{\bar x}, \mathbf{x}\in\re^{nm}$ and
$\bar x, x\in\re^n$.
Consequently, the sequence $\{x^k\}\subset\re^n$ generated by ACRM
is given by
\begin{equation}\label{equ270}
x^{k+1}= T_{\mathbf{S}_k,\mathbf{D}}(x^k).
\end{equation}
In practice, the sets $S^k_i$ will be halfspaces, so that the projections onto the sets $S^k_i$ are easily computable. The same holds for the operators $R_{\textbf W}, R_{{\textbf S}_k}$, and consequently for the operator $T_{{\textbf S}_k,{\textbf D}}$. This allows us to obtain a closed formula for $x^{k+1}$ in terms of $x^k$ in \eqref{eq.ACRM}. This is the case, for instance, of Algorithm PACA (see \cite{Behling:PACA}), which, using suitably perturbed separating operators, generates a sequence that is finitely convergent to a solution of CFP, under a Slater condition.

In this paper, we focus on the variational inequality problem $\VIP(F,C)$ with a feasible set of the form  $C \coloneqq \cap_{i=1}^m C_i$, where each $C_i\subset\re^n$ is closed and convex, \textit{i.e.}, we deal with the problem:
\begin{equation}
    \label{eq.VIP2}    
    {\rm find}~x^*\in C \\
    \,\,\,{\rm such\,\,\,that}\,\,\, \langle F(x^*), x-x^*\rangle \geq 0\,\,\,\forall x\in C.
\end{equation}

\subsection{Our contributions and assumptions}
In this work we propose two different algorithms, depending on the monotonicity and continuity properties of $F$, that make use of the circumcenter machinery to accelerate convergence. The firt one (Algorithm \ref{alg.ACRM_paramonotone}) is an acceleration of the algorithm proposed in \cite{Bello:2010}, while the second one (Algorithm \ref{alg.ECM}) accelerates an algorithm proposed in \cite{Bello:2012} for monotone operators. Both algorithms (and our novel scheme as well) use approximate projections onto separating halfspaces of the sets $C_i$ instead of exact projections onto $C$. We prove that our methods generate sequences converging to a solution of $\VIP(F,C)$ under mild assumptions.

We next list several assumptions, which will be needed for the convergence analysis of each algorithm. Our first two assumptions are required for both algorithms.
\begin{assumption}
\label{ass.nonemptySOL}
    The solution set of problem \eqref{eq.VIP2}, \emph{\textit{i.e.}},  $\SOL(F,C)$, is nonempty.
\end{assumption}
\begin{assumption}
\label{as.sequence_beta_k} The numerical sequence $\{\beta_k\}$ of stepizes is square summable but its sum diverges; that is, it 
satisfies the following conditions 
        \begin{align}
            \label{eq.assumption_betak.1} \sum_{k=0}^{\infty} \beta_k &= \infty, \\
            \shortintertext{and}
            \label{eq.assumption_betak.2} \sum_{k=0}^{\infty} \beta_k^2 &< \infty.
        \end{align}
\end{assumption}
The next two assumptions are required in the analysis of Algorithm \ref{alg.ACRM_paramonotone}
\begin{assumption}\label{a4}
    The mapping $F$ is paramonotone on $C$ and continuous on $\re^n$.
\end{assumption}
\begin{assumption}\label{a5}
    The mapping $F$ is monotone on $C$ and continuous on $\re^n$.
\end{assumption}
\begin{assumption}
\label{assm.q}
The mapping $F$ satisfies the following {\rm coerciveness} condition:
there exist $\hat{z}\in C$ and a bounded set $W\subset\re^n$ such that 
$\langle F(x),x-\hat{z}\rangle\ge 0$ for all $x \notin W$.
\end{assumption}

 Some comments are in order regarding Assumption \ref{assm.q}. Coerciveness assumptions
 have been required in previous convergence analyzes of direct methods for variational inequalities, for instance in \cite{Fukushima:1986}, where a condition stronger than Assumption
 \ref{assm.q} is demanded, and in \cite{Bello:2010}, where Assumption \ref{assm.q} is used. 
 Our assumption holds immediately when $F$ is monotone and has a zero $\bar z$, in 
 which case we take $\hat z=\bar z$ and $W=\emptyset$, and also when $F$ is
 uniformly monotone with $\lim_{t\to\infty}\psi(t)/t=\infty$ (for more details, see \cite{Bello:2010},
 where other monotone operators satisfying this assumption are discussed).

We end this section with the following assumption, which is required in the analysis of Algorithm \ref{alg.ECM}.
\begin{assumption}
\label{ass.monotonicty}
    The mapping $F$ is monotone on $C$ and continuous on $\re^n$.
\end{assumption}

\section{Preliminaries}

We now present some previously published results which will be used in our analysis. 
 

\begin{proposition}
\label{prop.ProjectionTheorem}
    Assume that $K\subset \mathds{R}^n$ is a closed convex set. Then, for any $x\in \mathds{R}^n$ and any $y \in K$ we have the following: 
    \begin{listi}
        \item $\langle x-P_K(x),y-P_K(x) \rangle \leq 0$,
    \item  $\lV x-y\rV^2 \geq \lV x-P_K(x)\rV^2 + 
    \lV y-P_K(x)\rV^2$.
    \end{listi}
\end{proposition}
\begin{proof}
    See Lemma 2 in \cite{Bello:2010}.
\end{proof}

\begin{proposition}
\label{lem.projection_to_halfspace}
    If $H\coloneqq \{y\in \re^n : \scal{a}{y}\leq \alpha\}$, with $a\in \re^n$, $\alpha \in R$, then the orthogonal projection onto $H$ is given by
    \[
        \label{eq.def.projection_to_halfspace}
        P_H(x) = x- \frac{\max\{0,\alpha -\scal{a}{x}\}}{\|a\|^2}a.
    \]
\end{proposition}
\begin{proof}
    See Proposition 2.2 in \cite{Behling:PACA}.
\end{proof}

\begin{proposition}
\label{prop.Circumcenter_is_projection}
Suppose that $C\subset \mathds{R}^n$ is a closed convex set and $U\subset \mathds{R}^n$ is an affine subspace. Let $S \subset \mathds{R}^n$ be a separating operator for $C$ which satisfies $(i)-(iii)$ in Definition \ref{separ}. Define $H(x) \subset \mathds{R}^n$ as:
\begin{equation}
    \label{eq.prop.Circumcenter_is_projection.1}
    H(x) = \begin{cases}
        K, & {\rm if}\,\,\,x\in C, \\
        \{z\in \mathds{R}^n: \langle z-P_S(x), x-P_S(x) \rangle \leq 0\}, & {\rm otherwise}.
    \end{cases}
\end{equation}
Then, for all $x \in U$, we have
\begin{equation}
    \label{}
    T_{S,U}(x) = P_{H(x)\cap U}(x),
\end{equation}
where $T_{S,U}:\mathds{R}^n \rightarrow \mathds{R}^n$ is defined as in \eqref{eq.CRM} w.r.t. $S$ and $U$.
\end{proposition}
\begin{proof}
    See Proposition 3.3 in \cite{Behling:2022a}.
\end{proof}
Next, we present an explicit instance of the approximate circumcenter operator $T_{\mathbf{S},\mathds{D}}$, defined in \eqref{eq.ACRM}. First, we define a separating operator $S_i(y)\subset\re^n$ for each set $C_i$, assumed to be of the form: 
\begin{equation}\label{equu}
C_i=\{x\in \re^n:g_i(x)\le 0\},
\end{equation} 
with a convex $g_i:\re^n\to\re$. 
$S_i$ is defined as follows:
\begin{equation}
    \label{eq.Si}
    S_i(y) = \begin{cases}
        C_i &  {\rm if}\,\,  y\in C_i, \\
        \{z\in \re^n: \scal{u_i}{z-y}+g_i(x)\le 0\}, & {\rm otherwise},
    \end{cases}
\end{equation}
where $u_i$ is a subgradient of $g_i$ at $y$. Given 
${\mathbf y}=(y^1, \dots ,y^m)\in\re^{nm}$, we define
${\mathbf S}({\mathbf y})\subset\re^{nm}$
as
\begin{equation}\label{equ701}
{\mathbf S}({\mathbf y})=S_1(y^1)\times\dots\times S_m(y^m).
\end{equation}
We have the following result related to these separating operators.
\begin{proposition}
\label{prop.Sik_Ci}
Suppose that $C_i\subset\re^n$ and $S_i(y)\subset\re^n$ are defined as in \eqref{equu} and \eqref{eq.Si}. Then $C_i\subset S_i(y)$ for any $y\in\re^n$ and $S_i(y)$ is a separating operator for $C_i$ at the point $y$ that satisfies conditions (i)-(iii) in Definition \ref{separ}. Also,
${\mathbf S}({\mathbf y})$ is a separating operator for 
$C_1\times\dots\times C_m\subset\re^{nm}$ at ${\mathbf y}$.
\end{proposition}
\begin{proof}
    See Proposition 2.8 in \cite{Behling:better}.
\end{proof}
\begin{remark}
    Take $C_i$  as in \eqref{equu}, $T_{\mathbf{S}_k,\mathbf{D}}$ as in \eqref{eq.ACRM} and $\mathbf{S}_k:= S_1(x^k) \times \cdots\times S_m(x^k)$ where $S_i(x^k)$ is defined in \eqref{eq.Si}. By Proposition \ref{prop.Sik_Ci} and \ref{prop.Circumcenter_is_projection}, we immediately have $ T_{\mathbf{S}_k,\mathbf{D}}(x) =P_{\mathbf{S}_k \cap \mathbf{D}}(x) =  P_{\mathbf{H}_k \cap \mathbf{D}}(x)$, where $\mathbf{H}_k:=H_1(x^k) \times \cdots\times H_m(x^k)$ and $H_i(x^k)$ is defined in a similar way as $H(x)$ in \eqref{eq.prop.Circumcenter_is_projection.1}  w.r.t. $C_i$.
\end{remark}

\begin{proposition}
\label{prop.better_Rates}
    Assume that $K \subset\re^n$ is a convex and closed set and $U \subset \re^n$ is an affine subspace. Then 
    \begin{equation}
        \label{}
        \lV T_{K,U}(x) - y\rV\le\lV T_{{\rm MAP}}(x)-y\rV,
    \end{equation}
    for all $x\in U$, for all $y \in K \cap U$, here $T_{K,U}$ is the circumcenter operator as defined in \eqref{eq.CRM} and $T_{{\rm MAP}}=P_U \circ P_K$ is the alternating projection operator with respect to $K$ and $U$.
\end{proposition}
\begin{proof}
    See Proposition 2.7 (ii) in \cite{Behling:better}
\end{proof}

\begin{proposition}
    \label{prop.maximalMonotone}
    If $F:\re^n\to\re^n$ is continuous then $\SOL(T,C)$ is closed.
\end{proposition}
\begin{proof}
    This follows from the definition of $\VIP(F,C)$ and the continuity of $F$.
\end{proof}

\begin{proposition}
\label{prop.maximal+paramonotone}
    Let $F:\re^n\to\re^n$ be a continuous and paramonotone operator. Let $\{z^k\}$ be a bounded sequence such that all cluster points of $\{z^k\}$ belong to $C$. For each $x\in \SOL(F,C)$ define $\gamma_k \coloneqq  \langle F(x^k), z^k-x\rangle$. If for some $\bar x\in \SOL(F,C)$ there exists a subsequence $\{\gamma_{k_j}\}$ of $\{\gamma_k\}$ such that 
    $\lim_{j \to\infty}\gamma_{k_j}\le 0$, then there exists a cluster point of $\{z^{k_j}\}$ belonging to $\SOL(F,C)$.
\end{proposition}
\begin{proof}
    See Lemma 6 in \cite{Bello:2010}.
\end{proof}
The previous two propositions were stated and proved in \cite{Bello:2010} for point-to-set 
operators. We have rephrased their statements to fit our point-to-point setting.

\begin{proposition}
    \label{prop.b}
    Take $\{\varepsilon_k\}$, $\{v_k\}$ $\subset$ $\mathds{R}_+$ and $\mu \in [0,1)$. If the inequalities
    \begin{equation}
        \label{eq.prob.b}
        \varepsilon_{k+1} \leq \mu \varepsilon_k + v_k,~~k \in \mathds{N},
    \end{equation}
    hold and $\lim_{k \rightarrow \infty} v_k =0$, then $\lim_{k \rightarrow \infty} \varepsilon_k =0$.
\end{proposition}
\begin{proof}
    See Lemma 2 in \cite{Fukushima:1986}.
\end{proof}

\begin{proposition}
    \label{prop.c}
    Take $K\subset \mathds{R}^n$ closed and $z\notin K$. Let $\{z^k\} \subseteq \mathds{R}^n$ be such that $\lim_{k \rightarrow \infty} \|z^{k+1}-z^k\|=0$ and both $z$ and some point in $K$ are cluster points of $\{z^k\}$. Then there exist $\zeta >0$ and a subsequence $\{z^{k_j}\}$ of $\{z^k\}$ such that
    \begin{equation}
        \label{eq.prop.c.1}
        \dist(z^{k_j+1},K) > \dist(z^{k_j},K),
    \end{equation}
    and
    \begin{equation}
        \label{eq.prop.c.2}
        \dist(z^{k_j},K) > \zeta.
    \end{equation}
\end{proposition}
\begin{proof}
    See Lemma 3 in \cite{Bello:2010}.
\end{proof}

\begin{definition}
    We say $w\in \mathds{R}^n$ is a Slater point for $\SOL(F,C)$ if
    \begin{equation}
        \label{}
        g_i(w) <0,
    \end{equation}
    for all $i=1,\ldots,m$.
\end{definition}

\begin{definition}
\label{def.quasi_Fejer_monotone}
    Let $S \subset \mathds{R}^n$. A sequence $\{x^k\} \subset \mathds{R}^n$ is called quasi-Fej\'er convergent to $S$ if for all $x \in S$ there exist $k_0 \geq 0$ and a sequence $\{\delta_k\} \subset \mathds{R}^n_+$ such that $\sum_{k=0}^{\infty} \delta_k < \infty$ and
    \begin{equation}
        \label{}
        \|x^{k+1} - x\|^2 \leq \|x^k - x\|^2 +\delta_k,
    \end{equation}
    for all $k \geq k_0$.
\end{definition}

\begin{proposition}
    \label{prop.SOL}
    Let $F:\mathds{R}^n \rightarrow \mathds{R}^n$ be monotone and continuous and $C$ a closed and convex set. Then $\SOL(F,C)$ is closed and convex.
\end{proposition}
\begin{proof}
    See Lemma 2.4(ii) in \cite{Bello:2009}.
\end{proof}

\begin{proposition}
\label{prop.Minty}
    Consider the variational inequality problem $\rm{VIP}(F,C)$. If $F:\mathds{R}^n \rightarrow \mathds{R}^n$ is continuous and monotone, then
    \begin{equation}
        \label{}
        \SOL(F,C) = \{x \in C: \langle F(y),y-x \rangle \geq 0,~\forall y\in C\}.
    \end{equation}
\end{proposition}
\begin{proof}
    See Lemma 3 in \cite{Bello:2012}.
\end{proof}

\begin{proposition}
    \label{prop.Qausi_Fejermonotone}
    If $\{x^k\}$ is a quasi-Fej\'er convergent to $S$ then:
    \begin{listi}
        \item $\{x^k\}$ is bounded;
        \item $\{\|x^k - x\|\}$ converges for all $x\in S$;
        \item if all cluster point of $\{x^k\}$ belong to $S$, then the sequence $\{x^k\}$ converges to a point in $S$. 
    \end{listi}
\end{proposition}
\begin{proof}
    See Proposition 1 in \cite{Bello:2012}.
\end{proof}

\begin{proposition}
\label{prop.convergence_projected_sequence}
    If a sequence $\{x^k\}$ is quasi-Fej\'er convergent to a closed and convex set $S \subset \mathds{R}^n$ then
    \begin{listi}
        \item the sequence $\{\dist(x^k,S)\}$ is convergent;
        \item the sequence $\{P_S(x^k)\}$ is convergent. 
    \end{listi}
\end{proposition}
\begin{proof}
    See Lemma 2 in \cite{Bello:2012}.
\end{proof}

\begin{proposition}
\label{prop.elemntary}
    Let $\{u^k\} \subset \mathds{R}^n$ be a sequence which converges to $\bar{u}$. Take nonnegative real number $\zeta_{k,j} (k\geq 0,0\leq j \leq k)$ such that $\lim_{k \rightarrow \infty} \zeta_{k,j}=1$ for all $k$. Define
    \begin{equation}
        \label{eq:prop.elemntary.1}
        w^k\coloneqq \sum_{j=0}^k \zeta_{k,j} u^j.
    \end{equation}
    Then, $\{w^k\}$ also converges to $\bar{u}$.
\end{proposition}
\begin{proof}
    Elementary.
\end{proof}

\begin{lemma}
\label{lem.ECM.dist}
    Let $g_i:\mathds{R}^n \rightarrow \mathds{R}$ is convex function for all $i \in \{1,\ldots,m\}$ and $C_i\coloneqq \{x\in \mathds{R}^n: g_i(x) \leq 0\}$. Define $C\coloneqq \cap_{i=1}^m C_i$. Assume that there exists a Slater point $w \in C$ such that $g(w) <0$ where $g$ is defined in \eqref{eq.g}. Then, for all $y$ such that $g(y)>0$, we have
    \begin{equation}
        \label{}
        \dist(y,C) \leq \frac{\|y-w\|}{g(y)-g(w)} g(y).
    \end{equation}
\end{lemma}
\begin{proof}
    Take $w_{\lambda}\coloneqq  \lambda w +(1-\lambda) y$ with $\lambda = \frac{g(y)}{g(y)-g(w)}$. Note that $\lambda \in (0,1)$. Then
    \begin{equation}
        \label{eq:lem.ECM.dist.1}
        g(w_{\lambda}) = g(\lambda w + (1-\lambda)y) \leq 
        \lambda g(w) + (1-\lambda)g(y) = g(y) - \lambda (g(y)-g(w)) = 0,
    \end{equation}
    where the first inequality holds because $g$, as defined in \eqref{eq.g}, is a convex function. Thus, $w_{\lambda} \in C$ and
    \begin{equation}
        \label{eq:lem.ECM.dist.2}
        \dist(y,C) \leq \|y-w_{\lambda}\| = \|y-(\lambda w + (1-\lambda)y)\| = \lambda \|y-w\| = \frac{g(y)}{g(y) - g(w)} \|y-w\|.
    \end{equation}
\end{proof}

\section{A circumcenter method for the variational inequality problem with paramonotone operator}

In this section, we consider a circumcenter method, based on the \GT\ method, for solving $\VIP(F,C)$ with a paramonotone operator $F$ and a feasible set of the form $C=\cap_{i=1}^m C_i$.

Our proposed method is constructed upon  Algorithms 1 and 2 proposed by Bello-Cruz and Iusem~\cite{Bello:2010}, which are versions of the \GT\ method for point-to-set paramonotone operators. Our approach is to employ a circumcenter-scheme simliarly to what was introduced  \cite{Behling:2021} in order to accelerate the ones in \cite{Bello:2010}. 

\subsection{The algorithm}
We present in Algorithm~\ref{alg.ACRM_paramonotone} the one-step circumcenter method for solving 
$\VIP(F,C)$, with paramonotone $F$ and feasible set $C$ of the form $C=\cap_{i=1}^m C_i$.



\begin{algorithm}[H]
\caption{Approximate Circumcenter method for $\VIP(F,C)$ with paramonotone $F$}
\label{alg.ACRM_paramonotone}
\begin{algorithmic}
    \Require Sequence $\{\beta_k\}$ as defined in Assumption \ref{as.sequence_beta_k}..
    
    \State \textbf{Initialization:} Choose $x^0\in\re^n$.
    
    \For{$k = 0, 1, 2, \dots$} \Comment{Iterative step}
        \If{$F(x^k)=0$}
            \State \textbf{stop}
        \EndIf
        
        \State Define $\eta_k\coloneqq \max\{1,\lV F(x^k)\rV\}$.
        \State Choose a separating operator $S^i:\re^n\to{\mathcal P}(\re^n)$ for the set $C_i$.
        \State Define $S^i_k\coloneqq S^i(x^k)$ and ${\mathbf S}_k\coloneqq S^1_k\times\dots\times S^m$.
        
        \State Compute the next iterate:
        \begin{equation}
            \label{eq.alg.ACRM_paramomoton.1}
            x^{k+1} = T_{\textbf{S}_k,\textbf{D}}\left(x^k-\frac{\beta_k}{\eta_k}F(x^k)\right),
        \end{equation}
        {{where $T_{\textbf{S}_k,\textbf{D}}$ is the \textit{approximate circumcentered-reflection operator} (defined in \eqref{equ270}) evaluated at $x^k- \frac{\beta_k}{\eta_k}F(x^k)$, and the $\textbf{S}_k$ is determined by $x^k$.}}
    \EndFor
\end{algorithmic}
\end{algorithm}

We remark that when $F$ is the gradient of a smooth convex function $f:\re^n\to\re$ 
(in which case, as mentioned above, $F$ is indeed paramonotone), Algorithm \ref{alg.ACRM_paramonotone} can be seen as 
an upgraded Projected Gradient method, with the acceleration effect resulting from the 
circumcenter approach.  

 We also remark that the use of stepsizes satisfying 
 \eqref{eq.assumption_betak.1}-\eqref{eq.assumption_betak.2} (called sometimes {\it vanishing} or
 $\ell_2\setminus\ell_1$ stepsizes) is a consequence of two facts: we do not assume Lipschitz continuity of $F$, which precludes the option of an exogenous constant stepsize, and our weaker monotonicity assumptions on $F$ exclude the implementation of successful line-searches for the stepsize.

\subsection{Convergence Analysis}
\label{subsec.Convergence_Analysis}

For the convergence analysis of Algorithm \ref{alg.ACRM_paramonotone}, we assume in this subsection that $F$ satisfies Assumption~\ref{a4}. We start with some preliminary results.

\begin{proposition}
    \label{prop.a}
    Suppose that $C=\cap_{i=1}^m C_i$. Let $V\subset \re^n$ be a nonempty compact and convex set. Define ${\mathbf V}=:V\times \cdots \times V\subset\re^{nm}$. Then, there exists
    $\bar\tau\in [0,1)$ such that  
    \begin{equation}
        \label{eq.prop.a.1}
        \dist(T_{{\mathbf S}_k\cap{\mathbf X},{\mathbf D}}(x),C) \le\bar\tau\,\dist(x,C),
    \end{equation}
     for all $x \in V\setminus C$, where  $T_{{\mathbf S}_k\cap{\mathbf X},{\mathbf D}}$ is as \eqref{eq.ACRM} and ${\mathbf S_k}=S_k^1\times\dots\times S^m_k$, where $S^i_k=S_i(x^k)$ and 
     $S^i:\re^n\to {\mathcal P}(\re^n)$ is a separating operator for $C_i$ in the sense of Definition \ref{separ}.
\end{proposition}
\begin{proof}
    First, we have 
    \begin{equation}
        \label{eq.prop.a.2}
        \dist(T_{\mathbf{S}_k \cap \mathbf{V},\mathbf{D}}(x),C)\le \dist((P_{\textbf{D}}\circ P_{\mathbf{S}_k \cap \mathbf{V}})(x),C),
    \end{equation}
    according to Proposition \ref{prop.better_Rates} with 
    $K=\mathbf{S}_k\cap\mathbf{V}$ and $U=\mathbf{D}$. Now we invoke
    Definition \ref{separ}(iv) to establish that there exists 
    $\tau_i\in (0,1]$ such that 
    \begin{equation}\label{equ112}
    \dist(P_{S_i\cap V}(x),C)\le\tau_i\,\dist(x,C).
    \end{equation}
    Define $\bar\tau=\max{\tau_1, \dots ,\tau_m}$.
    Using the definition of 
    $\mathbf{D}$ in \eqref{eq.D}, we  have 
    \begin{equation}
        \label{eq.prop.a.3}
        \begin{split}
             \dist((P_{\textbf{D}}\circ P_{\mathbf{S}_k \cap \mathbf{V}})(x),C) & = \dist\left(\frac{1}{m} \sum_{i=1}^m P_{S_i\cap V} (x),C\right) \\
             & \le\frac{1}{m}\sum_{i=1}^m \dist(P_{S_i\cap V}(x),C) \\
             & \le\frac{1}{m}\sum_{i=1}^m\tau_i\,\dist(x,C) \\
             & \le\bar\tau\,\dist(x,C),
        \end{split}
    \end{equation}
    for all $x \in V\setminus C$. The first inequality in \eqref{eq.prop.a.3} holds by convexity of the distance function, the second one holds by \eqref{equ112} and the third one by the definition of $\bar\tau$. Using \eqref{eq.prop.a.2} and \eqref{eq.prop.a.3}, we complete the proof.
\end{proof}

\begin{lemma}
\label{lem.1}
    Let $C =\cap_{i=1}^m C_i$. Consider the sequence $\{x^k\}$ defined by Algorithm \ref{alg.ACRM_paramonotone}. If $x^{k+1} = x^k $ for some $k$, then $x^k \in {\SOL(F,C)}$.
\end{lemma}
\begin{proof}
    We suppose that $x^{k+1} = x^k$, so in the product space $\mathds{R}^{nm}$ we have $\mathbf{x}^{k+1} = \mathbf{x}^k$. Since $\mathbf{x}^{k+1} = T_{\textbf{S}_k,\textbf{D}}\left(x^k - \frac{\beta_k}{\eta_k}F(x^k)\right)$, we get from Proposition \ref{prop.Circumcenter_is_projection} that $\mathbf{x}^k \in  \mathbf{H}_k \subset \mathbf{S}_k \subset \mathds{R}^{nm}$. Therefore, $x^k \in S_i^k$ for all $i\in \{1,\ldots,m\}$. Given $i \in \{1,\ldots,m\}$, we have two cases according to the definition of $S_i^k$ in \eqref{eq.Si}. If $S_i^k=C_i$, then we have $x^{k+1} = x^k \in C_i$. Otherwise,  we have $g_i(x^k) \leq g_i(x^k) + \langle v^{k}_i,x^{k+1} -x^k \rangle \leq 0$, where $v^{k}_i\in \partial g_i(x^k)$ \textit{i.e.}, $x^k \in C_i$. By the above argument above, we conclude that for any $i\in \{1,\ldots,m\}$, we have $x^k \in C_i$. Therefore, $x^k \in C=\cap_{i=1}^m C_i$. Moreover, since $x^{k+1}$ is given by \eqref{eq.alg.ACRM_paramomoton.1}, using Proposition \ref{prop.ProjectionTheorem} (i) in the product space, with $x = \mathbf{x}^k - \frac{\beta_k}{\eta_k}T_{\mathbf{S}_k,\mathbf{D}}(\mathbf{x}^k)$, the approximate circumcenter operator coincides with the projection operator, in view of Proposition \ref{prop.Circumcenter_is_projection}, \textit{i.e.}, $T_{\mathbf{S}_k,\mathbf{D}}(x) =P_{\mathbf{S}_k \cap \mathbf{D}}(x)$, and $K = \mathbf{S}_k$. Hence, we obtain (first in the product space)
    \begin{equation}
        \label{eq.lem1.1}
        \left \langle \mathbf{x}^{k+1}- (\mathbf{x}^k - \frac{\beta_k}{\eta_k}F(\mathbf{x}^k)), \mathbf{z}-\mathbf{x}^{k+1} \right \rangle \geq 0~~\forall \mathbf{z} \in \mathbf{S}_k,
    \end{equation}
    and then, in $\re^n$,
    \begin{equation}
        \label{eq.lem1.2}
        m\left\langle x^{k+1}-(x^k-\frac{\beta_k}{\eta_k}F(x^k)), z- x^{k+1}\right\rangle\ge 0 ~~\forall z\in S_i^k, 
    \end{equation}
    for all $i \in \{1,\ldots,m\}$.
    Taking $x^{k+1} = x^k$ in \eqref{eq.lem1.2} and taking into account the facts that $\beta_k >0, \eta_k \geq 1$ for all $k$ and $z\in C=\cap_{i=1}^m C_i \subset S_i^k$ in Proposition \ref{prop.Sik_Ci}, we get $\langle F(x^k), z-x^k \rangle \geq 0$ for all $z\in C$. Then, we conclude that $x^k \in {\SOL(F,C)}$.
\end{proof}

The following technical lemma will be used for establishing boundedness of $\{x^k\}$.

\begin{lemma}
\label{lem.2}
    Take $\hat{z} \in C\coloneqq \cap_{i=1}^m C_i$ and $W\subset{\mathbb R}^n$ as in Assumption \ref{assm.q}. Let $\{x^k\}$ be a sequence generated by Algorithm \ref{alg.ACRM_paramonotone}. Choose $\lambda >0$ such that $\lV x^0 -\hat{z}\rV\le \lambda$ and $W\subset B(\hat{z},\lambda)$. Then
    \begin{listi}
        \item if $x^k\in W$ then 
        $\lV x^{k+1}-\hat{z}\rV^2\le\lambda^2+\beta_k^2+2\beta_k\lambda$,
        \item if $x^k\notin W$ then 
        $\lV x^{k+1}-\hat z\rV^2\leq\lV x^k-\hat z\rV^2+\beta_k^2$.
    \end{listi}
\end{lemma}
\begin{proof}
    Since $\hat{z} \in C$, we get from \eqref{eq.closedFormular.1}  that $\hat{z} = T_{\mathbf{S}_k,\mathbf{D}}(\hat{z})$. Then, in view of \eqref{eq.alg.ACRM_paramomoton.1} and Proposition \ref{prop.ProjectionTheorem}(i), we obtain
    \begin{equation}
        \label{equ20}
        \begin{split}
            \lV x^{k+1}-\hat z\rV^2 &= \lV T_{\mathbf{S}_k,\mathbf{D}}(x^k-
            \frac{\beta_k}{\eta_k} F(x^k))-\hat z\rV^2=\lV T_{\mathbf{S}_k,\mathbf{D}}\left(x^k- \frac{\beta_k}{\eta_k}F(x^k)\right)-T_{\mathbf{S}_k,\mathbf{D}}(\hat{z})\rV^2 \\
            &\le\lV \left( x^k-\frac{\beta_k}{\eta_k}F(x^k)\right)-\hat z\rV^2 \\
            &=\lV x^k -\hat z\rV^2 +\frac{\beta_k^2}{\eta_k^2} \lV F(x^k)\rV^2- 
            2\frac{\beta_k}{\eta_k}\langle F(x^k),x^k-\hat z\rangle \\
            &\le\lV x^k -\hat z\rV^2+\beta_k^2-2\frac{\beta_k}{\eta_k}\langle F(x^k),x^k-\hat z \rangle.
        \end{split}
    \end{equation}
    Now,
    \begin{listi}
        \item if $x^k \in W$, applying the Cauchy-Schwarz inequality in \eqref{equ20}, the definition of $\eta_k$ and the fact that $W\subset B(\hat z,\lambda)$, we conclude that
        \begin{equation}
            \label{equ30}
            \lV x^{k+1}-\hat z\rV^2\le\lV x^k-\hat z\rV^2+\beta_k^2+2\frac{\beta_k}{\eta_k} 
            \lV F(x^k)\rV \lV x^k-\hat z\rV\le\lambda^2+\beta_k^2+2\beta_k\lambda.
        \end{equation}
        \item If $x^k \notin W$, it follows from Assumption \ref{assm.q} that 
        $\langle F(x^k),x^k-\hat z\rangle\ge 0$, and we get from \eqref{equ30} that
        \begin{equation}
            \label{}
            \lV x^{k+1}-\hat z\rV^2\le\lV x^k-\hat z\rV^2+\beta_k^2,
        \end{equation}
        using the fact that $\frac{\beta_k}{\eta_k}>0$.
    \end{listi}
\end{proof}

\begin{lemma}
    \label{lem.3}
    Let $\{x^k\}$, $\{F(x^k)\}$ be the sequences generated by Algorithm \ref{alg.ACRM_paramonotone}. Then,
    \begin{listi}
        \item $\{x^k\}$ and $\{F(x^k)\}$ are bounded,
        \item $\lim_{k \rightarrow \infty}\dist(x^k,C) = 0$,
        \item $\lim_{k \rightarrow \infty}\|x^{k+1} - x^k\| = 0$,
        \item All cluster points of $\{x^k\}$ belong to $C$.
    \end{listi}
\end{lemma}
\begin{proof}
    \begin{listi}
        \item Take $\hat z$ and $W$ as in Assumption \ref{assm.q}, $\lambda>0$ such that $\lV x^0 -\hat z\rV\le\lambda$  and $W\subset B(\hat z,\lambda)$, and $\bar\beta>0$ such that $\beta_k\le\bar\beta$ for all $k$ (such a $\bar\beta$ exists because $\sum_{k=0}^{\infty}\beta_k^2<\infty$). Let $\sigma = \sum_{j=0}^{\infty} \beta_j^2$. Define $\bar{\lambda}\coloneqq (\lambda^2+\bar{\beta}\lambda+\sigma)^{\frac{1}{2}}$. We claim that
        \begin{equation}
            \label{eq.lem.3.1}
            \{x^k\} \subset B(\hat{z}, \bar{\lambda}).
        \end{equation}
        We proceed to prove the claim. If $x^k \in B(\hat{z}, \lambda)$, we have $x^k \in B(\hat{z},\bar{\lambda})$, since $\bar{\lambda} >\lambda$. Otherwise, \textit{i.e.}, if $x^k \notin B(\hat{z},\lambda)$, let $\ell(k) = \max\{\ell <k:x^{\ell} \in B(\hat{z},\bar{\lambda}) \}$. Observe that $\ell(k)$ is well-defined because $\|x^0 - \hat{z}\| \leq \lambda$, so that $x^0 \in B(\hat{z},\bar{\lambda})$. By Lemma \ref{lem.2}(i),
        \begin{equation}
            \label{eq.lem.3.2}
            \lV x^{\ell(k)+1}-\hat z\rV^2\le\lambda^2+\beta_{\ell(k)}^2+2\beta_{\ell(k)}\lambda \le\lambda^2+\bar\beta\lambda+\beta_{\ell(k)}^2.
        \end{equation}
        Iterating the inequality in Lemma \ref{lem.2}(ii), since $x^j \notin W$ for $j$ between $\ell(k) + 1$ and $k-1$, we obtain
        \begin{equation}
            \label{eq.lem.3.3}
            \lV x^k-\hat z\rV^2\le\lV x^{\ell(k)+1}-\hat z\rV^2+ 
            \sum_{j=\ell(k)+1}^{k-1}\beta_k^2.
       \end{equation}
        Combining \eqref{eq.lem.3.2} and \eqref{eq.lem.3.3}, we obtain
        \begin{equation}
            \label{eq.lem.3.4}
            \begin{split}
                \lV x^k-\hat z\rV^2&\le\lambda^2+2 \bar\beta\lambda+
                \sum_{j=\ell(k)}^{k-1}\beta_j^2\\
                & \leq \lambda^2 +2 \bar{\beta} \lambda +\sum_{j = 0}^{\infty}\beta_j^2 \\
                & = \lambda^2 +2 \bar{\beta} \lambda +\sigma \\
                & = \bar{\lambda}^2.
            \end{split}  
        \end{equation}
        Thus, $x^k \in B(\hat{z},\bar{\lambda})$, proving the  claim, and hence $\{x^k\}$ is bounded. For $\{F(x^k)\}$, we use the boundedness of $\{x^k\}$, which implies boundedness of $\{F(x^k)\}$, since $F$ is continuous on $\re^n$ by Assumption \eqref{a4}.
        \item For all $k\ge 1$ we have that
        \begin{equation}
            \label{eq.lem.3.5}
            \lV x^{k+1}-T_{\mathbf{S}_k,\mathbf{D}}(x^k)\rV= 
            \lV T_{\mathbf{S}_k,\mathbf{D}}\left(x^k-\frac{\beta_k}{\eta_k}F(x^k)\right)-T_{\mathbf{S}_k,\mathbf{D}}(x^k)\rV\leq\frac{\beta_k}{\eta_k} \lV F(x^k)\rV
            \le\beta_k, 
        \end{equation}
        using \eqref{eq.alg.ACRM_paramomoton.1} and nonexpansiveness of $T_{\mathbf{S}_k,\mathbf{D}}$ in the first inequality, and the fact that $\eta_k \geq \|F(x^k)\|$ for all $k$ in the second one. 

        Since we already proved that $x^k \in  B(\hat{z},\bar{\lambda})$ for all $k\geq 1$, we conclude that
        \begin{equation}
            \label{}
            T_{\mathbf{S}_k\cap\mathbf{X},\mathbf{D}}(x)= T_{\mathbf{S}_k,\mathbf{D}}(x),
        \end{equation}
        for any $x \in B(\hat{z},\lambda)\setminus C$. Now we invoke Proposition \ref{prop.a} with $V=B(\hat z,\bar\lambda)$ and conclude that there exists $\Tilde{\mu} \in [0,1)$ such that
        \begin{equation}
            \label{eq.lem.3.6}
            \dist(T_{\mathbf{S}_k,\mathbf{D}}(x),C)\leq\Tilde{\mu}\, \dist(x,C),
        \end{equation}
        for all $x\in B(\hat z,\bar\lambda)\setminus C$.

        By \eqref{eq.lem.3.1}, $\{x^k\}\subset B(\hat z,\bar\lambda)$, so that it follows from \eqref{eq.lem.3.6} that
        \begin{equation}
            \label{eq.lem.3.7}
            \dist(T_{\mathbf{S}_k,\mathbf{D}}(x^k),C)\leq\Tilde{\mu}\,\dist(x^k,C), 
        \end{equation}
        for all $k$ such that $x^k \notin C=\cap_{i=1}^m C_i$. If $x^k \in C$, \eqref{eq.lem.3.7} holds trivially. Otherwise, observe that
        \begin{equation}
            \label{eq.lem.3.8}
            \begin{split}
                \dist(x^{k+1},C) &\le\lV x^{k+1}-T_{\mathbf{S}_k,\mathbf{D}}(x^k)\rV + 
                \dist(T_{\mathbf{S}_k,\mathbf{D}}(x^k),C) \\
                &\leq\beta_k+\Tilde{\mu}\, \dist(x^k,C),
            \end{split}
        \end{equation}
        using \eqref{eq.lem.3.8} and \eqref{eq.lem.3.6} in the second inequality. Therefore, using Proposition \ref{prop.b} with $v_k = \beta_k$, $\mu = \Tilde{\mu}$ and $\varepsilon_k = \dist(x^k,C)$, we obtain
        \begin{equation}
            \lim_{k\to\infty} \dist(x^k,C) = 0,
        \end{equation}
        establishing (ii).
        \item Using \eqref{eq.lem.3.5}, we get
        \begin{equation}
            \label{eq.lem.3.9}
            \lV x^{k+1}-x^k\rV\le
            \lV x^{k+1}-T_{\mathbf{S}_k,\mathbf{D}}(x^k)\rV+ 
            \lV T_{\mathbf{S}_k,\mathbf{D}}(x^k)-x^k\rV\le\beta_k+ 
            \dist(x^k,C).
        \end{equation}
        Since $\lim_{k\to\infty}\beta_k=0$ by \eqref{eq.assumption_betak.1}, it follows from (ii) and \eqref{eq.lem.3.9} that $\lim_{k\to\infty}\lV x^{k+1}-x^k\rV$.
        \item Follows from (ii). 
    \end{listi}
\end{proof}

Paramonotonicity of $F$ is used for the first time in this section in the following theorem.

\begin{theorem}
    \label{thm.CARM_paramonotone}
    If $F$ is paramonotone and $\SOL(F,C)\ne\emptyset$, then all cluster points of any sequence $\{x^k\}$ generated by Algorithm \ref{alg.ACRM_paramonotone} solve $\VIP(F,C)$.
\end{theorem}
\begin{proof}
    Let $\{x^k\}$, $\{F(x^k)\}$ be sequences generated by Algorithm \ref{alg.ACRM_paramonotone}. Define $\gamma_k: {\mathbb R}^n\to{\mathbb R}$ as
    \begin{equation}
        \label{eq.thm.CARM_paramonotone.1}
        \gamma_k(x)\coloneqq  \langle F(x^k),x^k -x\rangle.
    \end{equation}
    Take any $\bar x\in \SOL(F,C)$. From Proposition \ref{prop.ProjectionTheorem}(i) and the definition of $\eta_k$, we get that
    \begin{equation}
        \label{eq.thm.CARM_paramonotone.2}
        \begin{split}
            \lV x^{k+1}-\bar x\rV^2 & = \lV T_{\mathbf{S}_k,\mathbf{D}}\left(x^k-\frac{\beta_k}{\eta_k}F(x^k)\right)-T_{\mathbf{S}_k,\mathbf{D}}(\bar x)\rV^2 \\
            & \le\lV\left(x^k-\frac{\beta_k}{\eta_k}\right)-\bar x\rV^2 \\
            & =\lV x^k-\bar x\rV^2+\frac{\beta_k^2}{\eta_k^2}\lV F(x^k)\rV^2-2\frac{\beta_k}{\eta_k}\langle F(x^k),x^k-\bar x\rangle \\
            & \le\lV x^k-\bar x\rV^2-\beta_k\left(2\frac{\gamma_k(\bar x)}{\eta_k}-\beta_k\right).
        \end{split}
    \end{equation}
    First we prove that $\{x^k\}$ has cluster points in  $\SOL(F,C)$ Since $\{(x^k\},\{F(x^k))\}$ are bounded by Lemma \ref{lem.3}, in view of Proposition \ref{prop.maximal+paramonotone} it suffices to prove that $\{\gamma_k(\bar x)\}$ has a nonpositive cluster point. Assume that this is not true. Then, our assumption implies that $\{\gamma(\bar x)\}$ must be bounded away from zero for large $k$, \textit{i.e.}, there exist $\bar{k}$ and $\rho >0$ such that $\gamma_k(\bar x)\ge\rho$ for all $k\ge\bar k$. Since $\{F(x^k)\}$ is bounded, there exists $\theta >1$ such that 
    $\lV F(x^k)\rV\leq\theta$ for all $k$. Therefore,
    \begin{equation}
        \label{}
        \eta_k = \max\{1,\lV u^k\rV\}\le\max\{1,\theta\} = \theta,
    \end{equation}
    for all $k$. Thus, we can find $\bar\rho>0$ such that
    \[
        \frac{\gamma_k(\bar x)}{\eta_k}\geq\frac{\gamma_k(\bar x)}{\theta} >\bar\rho,
    \]
    and hence, in view of \eqref{eq.thm.CARM_paramonotone.2}, we obtain
    \begin{equation}
        \label{eq.thm.CARM_paramonotone.3}
        \lV x^{k+1}-\bar x\rV^2\leq\lV x^k-\bar x\rV^2-\beta_k(2\rho-\beta_k).
    \end{equation}
    Since $\lim_{k\to\infty}\beta_k=0$ by \eqref{eq.assumption_betak.1}, there exists $k'\ge\bar k$ such that $\beta_k\leq\bar\rho$ for all $k\ge k'$. So, we get from \eqref{eq.thm.CARM_paramonotone.3}, that, for all $k\ge k'$,
    \begin{equation}
        \label{eq.thm.CARM_paramonotone.4}
        \bar\rho\beta_k\le\lV x^k -\bar x\rV^2-\lV x^{k+1}-\bar x\rV^2.
    \end{equation}
    Summing \eqref{eq.thm.CARM_paramonotone.4} with $k$ between $k'$ and $q$, we obtain:
    \begin{equation}
        \label{eq.thm.CARM_paramonotone.5}
        \begin{split}
            \bar\rho\sum_{k =k'}^{q}\beta_k &\leq\sum_{k =k'}^{q} \left(\lV x^k-\bar x\rV^2- 
            \lV x^{k+1}-\bar x\rV^2\right) \\
            &\le\lV x^{k'}-\bar x\rV^2-\lV x^{q+1}-\bar x\rV^2\\
            &\le\lV x^{k'}-\bar x\rV^2.
        \end{split}
    \end{equation}
    Taking limits in \eqref{eq.thm.CARM_paramonotone.5} with $q\to\infty$, we contradict assumption \eqref{eq.assumption_betak.1}. Thus, there exists a cluster point of 
    $\{x^k\}$ belonging to  $\SOL(F,C)$.

    Finally, we prove that all cluster points of $\{x^k\}$ belong to  $\SOL(F,C)$. Suppose that 
    $\{x^k\}$ has a cluster point $z\notin \SOL(F,C)$. Since  $\SOL(F,C)$ is closed by Proposition \ref{prop.maximalMonotone}, and $\lim_{k\to\infty}\lV x^{k+1}-x^k\rV=0$ by Lemma \ref{lem.3}(iii), we invoke Proposition \ref{prop.c} to obtain a subsequence 
    $\{x^{j_k}\}$ of $\{x^k\}$ and a real number $\sigma >0$ such that
    \begin{equation}
        \label{eq.thm.CARM_paramonotone.6}
        \dist\left(x^{j_k},\SOL(F,C)\right)>\sigma,
    \end{equation}
    and
    \begin{equation}
        \label{eq.thm.CARM_paramonotone.7}
        \dist\left(x^{j_k+1},\SOL(F,C)\right)>\dist\left(x^{j_k},\SOL(F,C)\right).
    \end{equation}
    Take $\gamma_k$ as defined in \eqref{eq.thm.CARM_paramonotone.1}. 
    Note that for any $\bar x\in\SOL(F,C)$, $\{\gamma_k(\bar x)\}$ is bounded 
    by Lemma \ref{lem.3}(ii). Define $\gamma:\SOL(F,C)\to\mathds{R}$ as
    \begin{equation}
        \label{eq.thm.CARM_paramonotone.8}
        \gamma(x)\coloneqq \liminf_{k \rightarrow \infty} \gamma_{j_k}(x).
    \end{equation}
    We claim that $\gamma(x) >0 $ for all $x\in \SOL(F,C)$. Otherwise, by Proposition \ref{prop.maximal+paramonotone} $\{x^{j_k}\}$ has a cluster point in $\SOL(F,C)$, in contradiction with \eqref{eq.thm.CARM_paramonotone.6}. We next claim that $\gamma$ is continuous on $\SOL(F,C)$. Take $x,x'\in \SOL(F,C)$. Note that 
    \[
        \gamma_{j_k}(x^k) = \langle u^{j_k},x^{j_k}-x\rangle = \langle u^{j_k},x^{j_k}-x'\rangle + \langle u^{j_k},x'-x\rangle \leq \gamma_{j_k}(x') +\theta \lV x-x'\rV.  
    \]
    Thus, $\gamma(x)\le\gamma(x') + \theta \lV x-x'\rV$, where $\theta$ is an upper bound of 
    $\{\lV u^k\rV\}$. Reversing the role of $x,x'$, we obtain $\lv \gamma(x)-\gamma(x')\rv\le \theta\lV x-x'\rV$, establishing the claim.

    Let $V$ be the set of cluster points of $\{x^k\}$. We have shown above that $V \cap {\SOL(F,C)} \neq \emptyset$. Since $\{x^k\}$ is bounded, $V$ is compact and so is $V \cap \SOL(F,C)$. It follows that $\gamma$ attains its minimum on $V\cap\SOL(F,C)$ at some $x^*$, so that $\gamma(x)\ge\gamma(x^*)>0$ for all $x\in V\cap \SOL(F,C)$, using the claim above.

    Take $\hat k$ such that
    \begin{equation}
        \label{eq.thm.CARM_paramonotone.9}
        \gamma_{j_k}(x)\ge\frac{\gamma(x^*)}{2},
    \end{equation}
    and
    \begin{equation}
        \label{eq.thm.CARM_paramonotone.10}
        \beta_{j_k}<\frac{\gamma(x^*)}{\theta},
    \end{equation}
    for all $k\ge\hat k$. Note that $\hat k$ exists because, for large enough $k$, \eqref{eq.thm.CARM_paramonotone.9} holds by virtue of \eqref{eq.thm.CARM_paramonotone.8}, and \eqref{eq.thm.CARM_paramonotone.10} holds because $\lim_{k\to\infty}\beta_k=0$. In view of \eqref{eq.thm.CARM_paramonotone.2}, we get, for all $x\in V\cap\SOL(F,C)$ and all $k\ge \hat k$,
    \begin{equation}
        \label{eq.thm.CARM_paramonotone.11}
        \begin{split}
            \lV x^{j_k+1}-x\rV^2 &\le\lV x^{j_k}-x\rV^2-\beta_{j_k}\left(2\frac{\gamma_{j_k}(x)}{\eta_{j_k}}-\beta_{j_k}\right) \\
            &\le\lV x^{j_k}-x\rV^2-\beta_{j_k}\left(2\frac{\gamma_{j_k}(x)}{\theta}-\beta_{j_k} \right) \\
            &<\lV x^{j_k}-x\rV^2,
        \end{split}
    \end{equation}
    using \eqref{eq.thm.CARM_paramonotone.9} in the second inequality and \eqref{eq.thm.CARM_paramonotone.10} in the third one. It follows that 
    \[
        \dist(x^{j_k+1},V\cap\SOL(F,C))\le\dist(x^{j_k},V\cap\SOL(F,C)),
    \]
    for all $k\ge\hat k$, in contradiction with \eqref{eq.thm.CARM_paramonotone.7}. The contradiction arises from assuming that $\{x^k\}$ has cluster points outside $\SOL(F,C)$, and therefore all cluster points of $\{x^k\}$ solve $\VIP(F,C)$.
\end{proof}

We summarize the convergence properties of the sequence generated by Algorithm \ref{alg.ACRM_paramonotone} in the following corollary.

\begin{corollary}
    \label{}
    If $F$ and $C$ satisfy Assumptions \ref{ass.nonemptySOL}, \ref{a4} and \ref{assm.q}, (\textit{i.e.} 
    ${\SOL(F,C)}\neq\emptyset$ and $F$ is continuous and paramonotone),  
    then any sequence $\{x^k\}$ generated by Algorithm \ref{alg.ACRM_paramonotone} is bounded, $\lim_{k\to\infty}\lV x^{k+1}-x^k\rV=0$, and all cluster points of $\{x^k\}$ belong to $\SOL(F,C)$. If $\VIP(F,C)$ has a unique solution, then the whole sequence $\{x^k\}$ converges to it.
\end{corollary}
\begin{proof}
    It follows from Lemma \ref{lem.3}(i) and (iii), and Theorem \ref{thm.CARM_paramonotone}.
\end{proof}

\subsection {Algorithm \ref{alg.ACRM_paramonotone} when the $C_i$'s are sublevel sets of convex functions}

In this subsection, we consider the case when $C=\cap_{i=1}^m C_i$ and
each set $C_i$ is of the form 
\begin{equation}\label{equal}
C_i=\{x\in\re^n:g_i(x)\le 0\}, 
\end{equation}
where each $g_i:\re^n\to\re$ is convex. In principle, this form
entails no loss of generality, because we can take
$g_i(x)=\lV P_{C_i}(x)-x\rV^2$, but this case turns out to be
computationally advantageous only when we have easily computable formulae for $g_i$ and its subdifferential. We remark that this is indeed the case in a large number of applications. When the $C_i$'s are as in \eqref{equal}, we can use a very convenient
operator $S^i$, which leads to explicit closed-form expressions for the iteration of Algorithm \ref{alg.ACRM_paramonotone}.

We define a separating operator for the set $C_i$ as
\begin{equation}\label{equ114}
    S^i(x) = \{y\in\re^n: g_i(x)+\langle u_i,y-x\rangle\le 0\},
\end{equation}
where $u_i$ is a subgradient of $g_i$ at $x$. 

In connection with the sequence $\{x^k\}$ generated by Algorithm \ref{alg.ACRM_paramonotone} we
define $S^i_k\coloneqq S_i(x^k)$, \textit{i.e.},
\begin{equation}\label{equ115}
S^i_k=\{y\in\re^n: g_i(x^k)+\langle u^{k}_i,y-x^k\rangle\le 0\},
\end{equation}
where now $u^{k}_i$ is a subgradient of $g_i$ at $x^k$. Finally, we 
define ${\mathbf S}_k\subset\re^{nm}$ as
\begin{equation}\label{equ116}
{\mathbf S}_k = S^1_k\times\dots\times S^m_k.
\end{equation}
This choice of the set ${\mathbf S}_k$ leads to the closed-form presentation of Algorithm \ref{alg.ACRM_paramonotone}, namely:
\begin{equation}
    \label{eq.closedFormular.1}
    x^{k+1} = T_{\textbf{S}_k,\textbf{D}}\left(x^k-\frac{\beta_k}{\eta_k}F(x^k)\right)=x^k-\frac{\beta_k}{\eta_k} F(x^k)-\alpha_k w^k,
\end{equation}
where 
\begin{equation}\label{equa2}
 \alpha_k = \frac{\sum_{i=1}^m\lV v^{k}_i\rV^2}{m\lV w^k\rV^2}, 
 \end{equation}
 \begin{equation}\label{equa3}
 w^k =\frac{1}{m}\sum_{i=1}^m v^{k}_i, 
 \end{equation}
 \begin{equation}\label{equa4}        
 v^{k}_i = \left[\frac{\max\left\{0,g_i\left(x^k-\frac{\beta_k}{\eta_k}F(x^k)\right)\right\}}{\lV u^{k}_i\rV}\right]u^{k}_i, \text{ for } i=1,2,\dots,m,
\end{equation}
with $u^{k}_i\in\partial g_i\left(x^k-\frac{\beta_k}{\eta_k}F(x^k)\right)$.

\begin{remark}\label{closedf}
We now make the following comments on formulae \eqref{eq.closedFormular.1}, \eqref{equa2}, \eqref{equa3}, and \eqref{equa4}.
\begin{listi} 
    \item The formula for $\alpha_k$ follows from the definition of the circumcenter, using Proposition \ref{lem.projection_to_halfspace}
    in \eqref{equa4}; the computation of $\alpha_k$ is similar to the one presented in \cite{Behling:PACA}, in the definition of the PACA Algorithm for solving CFP. In fact, the main difference between our closed formulae above and those in \cite{Behling:PACA} is the presence of the term $(\beta/\eta_k) F(x^k)$ in \eqref{eq.closedFormular.1}, related to the operator $F$, which is specific to $\VIP(F,C)$ and is absent in CFPs.
    \item The closed formulae  \eqref{eq.closedFormular.1}-\eqref{equa4} were not used in the convergence analysis of Subsection \ref{subsec.Convergence_Analysis}, which deals with general separating operators satisfying the conditions in Definition \ref{separ}, but we note that they are essential for the effective implementation of the algorithm.
    \item These  closed formulae hold only when the sets $C_i$ are of the form 
    $C_i=\{x\in\re^n:g_i(x)\le 0\}$ for convex functions $g_i$. The iteration formula \eqref{eq.alg.ACRM_paramomoton.1}, which is the one used in our convergence analysis, covers the case of general sets $C_i$ and the use of other separating operators ${\textbf S}_k$, including even the exact projection case, for which $S_i^k=C_i$ for all $i,k$.
    \item Observe that $\partial g_i(x)\ne\emptyset$ for all $x\in\re^n$, because we assume that $g_i$ is convex and $\dom(g_i) =\re^n$ for all $i$. 
\end{listi}    
\end{remark}

In order to ensure that the convergence analysis in Subsection
\ref{subsec.Convergence_Analysis} holds for ${\mathbf S}_k$ as defined
by \eqref{equ115} and \eqref{equ116}, we must prove that the separating operator $S^i$ defined by \eqref{equ114} indeed satisfies the conditions in Definition \ref{separ}. We establish this in the next two propositions.

\begin{proposition}
\label{prop.sufficient_process}
Let $f:\re^n\to\re$ be a convex function and $V\subset\re^n$ be a nonempty, convex, and compact set.
Define $C=\{z\in\re^n:f(z)\le 0\}$. Given $x\in\re^n$ and $v\in\partial f(x)$, define
$\bar S(x,u)\coloneqq \{z\in\re^n : f(x)+\langle u,z-x\rangle\le 0\}$. Then, there exists $\kappa \in [0,1)$ such that 
    \begin{equation}
        \label{}
        \dist(P_{\bar S(x,v)\cap X}(x),C)\le\kappa \,\dist(x,C),
    \end{equation}
    for all $x\in V$ and all $u\in\partial f(x)$.
\end{proposition}
\begin{proof}
    See Lemma 4 in \cite{Fukushima:1983}.
\end{proof}

\begin{proposition}\label{pp1}
Given a convex function $g:\re^n\to\re$, let $C\coloneqq \{x\in\re^n:g(x)\le 0\}$.
Define $S(x)\subset\re^n$ by
$$
S(x)\coloneqq \{y\in\re^n: g(x)+\langle u,y-x\rangle\le 0\},
$$
where $u\in\re^n$ is a subgradient of $g$ at $x$.
Then $S$ is a separating operator for $C$ in the sense of Definition \ref{separ}.
\end{proposition}
\begin{proof}
The fact that $S$ satisfies conditions (i)-(iii) in Definition \ref{separ} has been proved in  Proposition 2.8 of \cite{Behling:2022a}. We now verify condition (iv). The result is an immediate consequence
of Proposition \ref{prop.sufficient_process}. 
\end{proof}

Proposition \ref{pp1} establishes that ${\mathbf S}_k$, as
defined in \eqref{equ116}, is compatible with the convergence analysis of
Algorithm \ref{alg.ACRM_paramonotone} carried out in Subsection \ref{subsec.Convergence_Analysis}. 

We also mention that an argument similar to the one in the proof of Proposition \ref{pp1} establishes
that the operator ${\mathbf S}:\re^{nm}\to{\mathcal P}(\re^{nm})$ given by
${\mathbf S}({\mathbf x})= (S^1(x^1), \dots, S^m(x^m))$, where
${\mathbf x}=(x^1,\dots ,x^m)$ with $x^i\in\re^n\,\,(1\le i\le m)$,
and $S^i$ is defined as in \eqref{equ114}, satisfies conditions (i)-(iv) in
Definition \ref{separ}. The proof for conditions (i)-(iii) can be found in \cite{Behling:2022a}. However, this result is not needed for our convergence analysis.


\section{An explicit circumcenter method for the monotone variational inequality problem}

In this section, we introduce an algorithm that unlike the algorithm in Section 3, requires only the assumption of monotonicity instead of paramonotonicity. Our proposal is formally described in Algorithm \ref{alg.ECM} below. It is an acceleration of the method developed by Bello-Cruz and Iusem \cite{Bello:2012}. The main difference between our proposal is the use of circumcenters. 

To the best of our knowledge, there are few algorithms that solve a VIP with a monotone operator using approximate projections that may produce infeasible iterates.
An algorithm with inexact projections was proposed in \cite{Millan:2024}, but those projections remain feasible and still require solving an auxiliary problem at each iteration.
In contrast, the approximate projections we implement in this paper are computationally cheaper, even though they may yield infeasible iterates.

\begin{algorithm}
\caption{Explicit Circumcenter Method (ECM) for $\VIP(F,C)$ with monotone operators}
\label{alg.ECM}
\begin{algorithmic}
    \Require Slater point  $w \in \re^n$, that is, a point such that  $g_i(w) < 0$ for all $i \in \{1,\ldots,m\}$; exogenous constant $\theta > 0$; sequence $\{\beta_k\}$ as defined in Assumption \ref{as.sequence_beta_k}.

    \State \textbf{Initialization:} Set $x^0 := 0$ and choose $z^0 \in \mathds{R}^n$.

    \For{$k = 0, 1, 2, \dots$} \Comment{Iterative Step}
        \State Define $g(x) \coloneqq \max_{1 \leq i \leq m} g_i(x)$.
        
        \If{$g_i(z^k) \leq 0$ for all $i \in \{1,\ldots,m\}$}
            \State $\Tilde{y}^k \gets z^k$
        \Else \Comment{Perform inner loop}
            \State Set $y^{k,0} \gets z^k$ and $j \gets 0$.
            \Repeat
                \State Compute next inner iterate:
                \begin{equation}
                    \label{eq.inner_loop}
                    y^{k,j+1} := T_{\mathbf{S}_k,\mathbf{D}}(y^{k,j}).
                \end{equation}
                \State Check stopping condition for $j$:
                \begin{equation}
                    \label{eq.jk}
                    \text{Condition: } \frac{g(y^{k,j})\|y^{k,j}-w\|}{g(y^{k,j})-g(w)} \leq \theta \beta_k,
                \end{equation}
                where
                \begin{equation}
                    \label{eq.g}
                    g(x) = \max_{1 \leq i \leq m} g_i(x).
                \end{equation}
                \If{Condition is not met} $j \gets j+1$ \EndIf
            \Until{Condition \eqref{eq.jk} is met}
            \State Set $j(k) \gets j$ and define:
            \begin{equation}
                \label{eq.tilde_y}
                \Tilde{y}^k := y^{k,j(k)}.
            \end{equation}
        \EndIf

        \State Define $\eta_k \coloneqq \max\{1, \|F(\Tilde{y}^k)\|\}$. 
        \State Compute next $z$ iterate:
        \begin{equation}
            \label{eq.zk}
            z^{k+1} := T_{\mathbf{S}_k,\mathbf{D}} \left( \Tilde{y}^k - \frac{\beta_k}{\eta_k} F(\Tilde{y}^k) \right).
        \end{equation}

        \If{$z^{k+1} = \Tilde{y}^k$} 
            \State \textbf{stop}
        \EndIf

        \State Update partial sum and average iterate:
        \begin{equation}
            \label{eq.sigma_k}
            \sigma_k := \sum_{j=0}^k \frac{\beta_j}{\eta_j}. \quad (\text{accumulated})
        \end{equation}
        \begin{equation}
            \label{eq.ECM.xk}
            x^{k+1} := \left( 1-\frac{\beta_k}{\eta_k \sigma_k}\right) x^k + \frac{\beta_k}{\eta_k \sigma_k} \Tilde{y}^k.
        \end{equation}
        
    \EndFor
\end{algorithmic}
\end{algorithm}

We assume that a Slater point $w\in \mathds{R}^n$ is available. This assumption is needed only in the inner loop \eqref{eq.inner_loop} in order to certify whether the current iterate in the loop is sufficiently close to the feasible set.

\begin{remark}
The computation in the inner loop consists solely of circumcenter steps, which require projections onto halfspaces and are therefore quite inexpensive. Even if the evaluation of $F$ is costly, it is performed only once per iteration, as in the algorithms proposed in \cite{Malitsky:2015}.
\end{remark}

\subsection{Analysis of the direct method}

In this subsection we assume that $F,C$ satisfy Assumptions \ref{ass.nonemptySOL} and \ref{ass.monotonicty}, \textit{i.e.} $\VIP(F,C)$ has solutions and $F$ is monotone on $C$ and continuous
on $\re^n$.

Observe that $\partial g_i(x) \neq \emptyset$ for $x \in \mathds{R}^n$ for all $i \in \{1,\ldots,m\}$, because we assume that $g_i$ is convex and $\dom(g_i) = \mathds{R}^n$, for all $i \in \{1,\ldots,m\}$. First, we establish that Algorithm \ref{alg.ECM} is well-defined.
\begin{lemma}
    \label{lem.ECM.1}
    Take $C_i$, $\Tilde{y}^k$, $z^k$ and $x^k$ defined by \eqref{equu}, \eqref{eq.tilde_y}, \eqref{eq.zk} and \eqref{eq.ECM.xk} respectively. Then
    \begin{listi}
    \item $g(y^{k,j}) +\langle v^{k,j}, y^{k,j+1} - y^{k,j} \rangle \leq 0$, where $g$ is defined in \eqref{eq.g};
        \item If $z^{k+1} = \Tilde{y}^k$, then $\Tilde{y}^k \in \SOL(F,C)$;
        \item $j(k)$ is well-defined. 
    \end{listi}
\end{lemma}
\begin{proof}
\begin{listi}
\item
By \eqref{eq.inner_loop} and Proposition \ref{prop.Circumcenter_is_projection}, we know that $y^{k,j+1} =T_{\mathbf{S}_k,\mathbf{D}}(y^{k,j})  \in S_i^k$ for all $i\in \{1,\ldots,m\}$. In view of the definition of $S_i^k$ and $g$, we know that (i) holds automatically. 
\item 
Suppose that $z^{k+1} = \Tilde{y}^k$. Then, by Proposition \ref{prop.Circumcenter_is_projection}, we get that $\mathbf{\Tilde{y}}^k \in \mathbf{H}_k \subset \mathbf{S}_k \subset \mathds{R}^{nm}$. 
Therefore, $\Tilde{y}^k \in S_i^k$ for all $i\in \{1,\ldots,m\}$. 
Given $i \in \{1,\ldots,m\}$, we have two cases, according to the definition of $S_i^k$ in \eqref{eq.Si}. If $S_i^k=C_i$, then we have $z^{k+1} = \Tilde{y}^k \in C_i$. Otherwise,  we have 
$$g_i(\Tilde{y}^k)= g_i(z^{k+1}) \leq g_i(\Tilde{y}^k) + \langle \tilde v_i^{k,j},z^{k+1} -\Tilde{y}^k \rangle \leq 0,$$
where $\tilde v_i^{k,j}\in \partial g_i(\Tilde{y}^k)$. The equality holds by the assumption, the first inequality follows from the subgradient inequality and the last one in consequence of the fact that $\tilde{y}^k \in S_i^k$. Therefore, $\Tilde{y}^k \in C_i$. By the above argument above, we conclude that for any $i\in \{1,\ldots,m\}$, we have $\Tilde{y}^k \in C$. Moreover, since $z^{k+1}$ is given by \eqref{eq.zk}, we can use the similar argument as in Lemma \ref{lem.1} to conclude that $\tilde{y}^k \in {\rm SOL}(F,C)$.
\item Assume by contradiction that 
$$\frac{g(y^{k,j})\|y^{k,j}-w\|}{g(y^{k,j})-g(w)} > \theta \beta_k,$$
for all $j$. Thus, we obtain an infinite sequence $\{y^{k,j}\}$ such that
    \begin{equation}
        \label{eq.lem.ECM.1.1}
        \liminf_{j \rightarrow \infty} \frac{g(y^{k,j})\|y^{k,j}-w\|}{g(y^{k,j})-g(w)} \geq \theta \beta_k >0.
    \end{equation}
Taking into account the inner loop in $j$ given in \eqref{eq.inner_loop} \textit{i.e.}, $y^{k,j+1} \coloneqq  T_{\mathbf{S}_k,\mathbf{D}}(y^{k,j})$ for each fixed $k$, we obtain, for each $x \in C$,
\begin{equation}
    \label{eq.lem.ECM.1.2}
    \begin{split}
    \|y^{k,j+1} - x\|^2 & = \|T_{\mathbf{S}_k,\mathbf{D}}(y^{k,j}) - T_{\mathbf{S}_k,\mathbf{D}}(x)\|^2  = \|P_{\mathbf{S}_k,\mathbf{D}}(y^{k,j}) - P_{\mathbf{S}_k,\mathbf{D}}(x)\| \\
    &  \leq \|y^{k,j} - x\|^2 - \|y^{k,j+1} - y^{k,j}\|^2 \\
    & \leq \|y^{k,j}-x\|^2,
    \end{split}
\end{equation}
where the first equality follows from Proposition \ref{prop.Circumcenter_is_projection} and the first inequality holds because of the firm nonexpansiveness of the projection operator. Thus, $\{y^{k,j}\}_{j \geq 0}$ is quasi-Fej\'er convergent to $C$, and hence it is bounded by Proposition \ref{prop.Qausi_Fejermonotone}. It follows that 
\begin{equation}
    \label{eq.lem.ECM.1.3}
    \tau\coloneqq  \frac{1}{-g(w)} \sup_{0\leq j \leq\infty} \|y^{k,j}-w\|,
\end{equation}
is finite, because $w$ is a Slater point. Also,
\begin{equation}
    \label{eq.lem.ECM.1.4}
    g(y^{k,j}) >0~for~all~j.
\end{equation}
Using \eqref{eq.lem.ECM.1.2}, we get
\begin{equation}
    \label{eq.lem.ECM.1.5}
    \lim_{j\rightarrow \infty} \|y^{k,j+1}-y^{k,j}\| = \lim_{j\rightarrow \infty} \|T_{\mathbf{S}_k,\mathbf{D}}(y^{k,j}) - y^{k,j}\| = 0.
\end{equation}
By Lemma \ref{lem.ECM.1}(i), we have
\begin{equation}
    \label{eq.lem.ECM.1.6}
    g(y^{k,j}) \leq \langle v^{k,j}, y^{k,j} - y^{k,j+1} \rangle \leq \|v^{k,j}\| \|y^{k,j} - y^{k,j+1}\|.
\end{equation}
Since $\{y^{k,j}\}_{j \geq 0}$ is bounded, $g$ is convex, and $\partial g$ is bounded on bounded sets, we obtain that $\{v^{k,j}\}_{j \geq 0}$ is bounded. In view of \eqref{eq.lem.ECM.1.5} and \eqref{eq.lem.ECM.1.6},
\begin{equation}
    \label{eq.lem.ECM.1.7}
    \liminf_{j \rightarrow \infty} g(y^{k,j}) \leq 0.
\end{equation}
It follows from \eqref{eq.lem.ECM.1.4} and \eqref{eq.lem.ECM.1.7} that
\begin{equation}
    \label{eq.lem.ECM.1.8}
    \begin{split}
        \liminf_{j \rightarrow \infty} \frac{g(y^{k,j})\|y^{k,j}-w\|}{g(y^{k,j})-g(w)} & \leq \liminf_{j \rightarrow \infty} \frac{g(y^{k,j})\|y^{k,j}-w\|}{-g(w)} \\
        & \leq \frac{1}{-g(w)} \sup_{ 0 \leq j \leq \infty} \|y^{k,j} - w\| \liminf_{j \rightarrow \infty} g(y^{k,j}) \\
        &  = \tau \liminf_{j \rightarrow \infty} g(y^{k,j}) \leq 0,
    \end{split}
\end{equation}
contradicting \eqref{eq.lem.ECM.1.1}. Therefore, $j(k)$ is well-defined.
\end{listi}
\end{proof}

\begin{lemma}
\label{lem.ECM.quasi_Fejermonotonicity}
If $\SOL(F,C)$ is nonempty, then $\{\Tilde{y}^k\}$ and $\{z^k\}$ are quasi-Fej\'er convergent to $\SOL(F,C)$.
\end{lemma}
\begin{proof}
    Observe that $\eta_k \geq \|F(\Tilde{y}^k)\|$ and $\eta_k \geq 1$ for all $k$ by the definition of $\eta_k$. Then, for all $k$,
    \begin{equation}
        \label{eq.lem.ECM.quasi_Fejermonotonicity.1}
        \frac{1}{\eta_k} \leq 1,
    \end{equation}
    and 
    \begin{equation}
        \label{eq.lem.ECM.quasi_Fejermonotonicity.2}
        \frac{\|F(\Tilde{y}^k)\|}{\eta_k} \leq 1.
    \end{equation}
    Take any $x^* \in \SOL(F,C)$, we get
    \begin{equation}
        \label{eq.lem.ECM.quasi_Fejermonotonicity.3}
        \begin{split}
            \|\Tilde{y}^k - x^*\|^2 & = \|y^{k,j(k)} - x^*\|^2 = \|T_{\mathbf{S}_k,\mathbf{D}}(y^{k,j(k-1)}) - T_{\mathbf{S}_k,\mathbf{D}}(x^*)\|^2 \leq \|y^{k,j(k)-1} - x^*\|^2 \\
            & = \|T_{\mathbf{S}_k,\mathbf{D}}(y^{k,j(k)-2}) - T_{\mathbf{S}_k,\mathbf{D}}(x^*)\|^2 \leq \|y^{k,j(k)-2}-x^*\|^2 \\
            & \cdots \\
            & \leq \|y^{k,1} - x^*\|^2 = \|T_{\mathbf{S}_k,\mathbf{D}}(y^{k,0}) - T_{\mathbf{S}_k,\mathbf{D}}(x^*)\|^2 \\
            &\leq \|y^{k,0}-x^*\|^2 = \|z^k - x^*\|^2.
        \end{split}
    \end{equation}
    Let $\Tilde{\theta}=1 + \theta \|F(x^*)\| \geq 1 + \theta \frac{\|F(x^*)\|}{\eta_k}$, by \eqref{eq.lem.ECM.quasi_Fejermonotonicity.1}. Then
    \begin{equation}
        \label{eq.lem.ECM.quasi_Fejermonotonicity.4}
        \begin{split}
            \|\Tilde{y}^{k+1} - x^*\|^2 & \leq \|z^{k+1} -x^*\|^2 = \left \| T_{\mathbf{S}_k,\mathbf{D}}(\Tilde{y}^k - \frac{\beta_k}{\eta_k}F(\Tilde{y}^k)) - T_{\mathbf{S}_k,\mathbf{D}}(x^*) \right\|^2 \\ & \leq \|\Tilde{y}^k - \frac{\beta_k}{\eta_k}F(\Tilde{y}^k) - x^*\|^2 \\
            & =\|\Tilde{y}^k - x^*\|^2 + \frac{\|F(\Tilde{y}^k)\|^2}{\eta_k^2} \beta_k^2 -2\frac{\beta_k}{\eta_k} \langle F(\Tilde{y}^k),\Tilde{y}^k -x^* \rangle \\
            & \leq \|\Tilde{y}^k - x^*\|^2 + \beta_k^2 - 2\frac{\beta_k}{\eta_k} \langle F(x^*),\Tilde{y}^k -x^* \rangle \\
            & = \|\Tilde{y}^k - x^*\|^2 + \beta_k^2 - 2\frac{\beta_k}{\eta_k} ( \langle F(x^*), \Tilde{y}^k - P_C(\Tilde{y}^k) \rangle + \langle F(x^*), P_C(\Tilde{y}^k) - x^* \rangle )\\
            & \leq \|\Tilde{y}^k - x^*\|^2 + \beta_k^2 - 2\frac{\beta_k}{\eta_k} \|F(x^*)\| \|P_C(\Tilde{y}^k)-\Tilde{y}^k\| \\
            & \leq |\Tilde{y}^k - x^*\|^2 + \Tilde{\theta}\beta_k^2 \\
            & \leq \|z^k - x^*\|^2,
        \end{split}
    \end{equation}
    using \eqref{eq.lem.ECM.quasi_Fejermonotonicity.3} in the last inequality, the nonexpansiveness of $T_{\mathbf{S}_k,\mathbf{D}}$ in the second one, the monotonicity of $F$ and \eqref{eq.lem.ECM.quasi_Fejermonotonicity.2} in the third one, the fact that $x^* \in \SOL(F,C)$ in the fourth one, the Cauchy-Schwartz inequality in the fifth one, Lemma \ref{lem.ECM.dist} and definition of $j(k)$ in the sixth one, and \eqref{eq.lem.ECM.quasi_Fejermonotonicity.3} in the last one.

    In view of Definition \ref{def.quasi_Fejer_monotone}, we conclude that $\{\Tilde{y}^k\}$ and $\{z^k\}$ are quasi-Fej\'er convergent to $\SOL(F,C)$.
\end{proof}

\begin{lemma}
\label{lem.ECM.feasibility_clusterPoints}
    Let $\{z^k\}$, $\{\Tilde{y}^k\}$ and $\{x^k\}$ be the sequences generated by Algorithm \ref{alg.ECM}. Assume that $\SOL(F,C)$ is nonempty. Then
    \begin{listi}
        \item $x^{k+1} = \frac{1}{\sigma_k} \sum_{k=1}^k \frac{\beta_j}{\eta_j}\Tilde{y}^j$;
        \item $\{\Tilde{y}^k\}$, $\{x^k\}$ and $\{F(\Tilde{y}^k)\}$ are bounded;
        \item $\lim_{k \rightarrow \infty} \dist(x^k,C) = 0$;
        \item all the cluster points of $\{x^k\}$ belong to $C$.
    \end{listi}
\end{lemma}
\begin{proof}
\begin{listi}
\item Note that \eqref{eq.sigma_k} and \eqref{eq.ECM.xk} can be rewritten as 
    \begin{equation}
        \label{eq.lem.ECM.feasibility_clusterPoints.1}
        \begin{split}
            \sigma_k x^{k+1}  = \left( \sigma_k - \frac{\beta_k}{\eta_k} \right)x^k + \frac{\beta_k}{\eta_k} \Tilde{y}^k & = \sigma_{k-1}x^k + \frac{\beta_k}{\eta_k}\Tilde{y}^k \\
            &  = \sigma_{k-2} x^{k-1} + \frac{\beta_{k-1}}{\eta_{k-1}} \Tilde{y}^{k-1} +\frac{\beta_k}{\eta_k} \Tilde{y}^k \\
            & \cdots \\
            & = \sigma_0 x^1 + \sum_{j=0}^k \frac{\beta_j}{\eta_j}\Tilde{y}^j \\
            &  = \sum_{j=0}^k \frac{\beta_j}{\eta_j}\Tilde{y}^j,
        \end{split}
    \end{equation}
    using the fact that $\sigma_0 = 0$, which follows from \eqref{eq.sigma_k}. We obtain the result after dividing both sides by $\sigma_k$. 
    
\item For $\{\Tilde{y}^k\}$ we use Lemma \ref{lem.ECM.quasi_Fejermonotonicity} and Proposition \ref{prop.Qausi_Fejermonotone}(i). For $\{F(\Tilde{y}^k)\}$, we use the boundedness of  $\{\Tilde{y}^k\}$ and the fact that $F$ is monotone and continuous. We use the boundedness of $\{\Tilde{y}^k\}$ and (i) in for $\{x^k\}$. 

\item By Lemma \ref{lem.ECM.dist} and \eqref{eq.jk}-\eqref{eq.tilde_y}, we have
    \begin{equation}
        \label{eq.lem.ECM.feasibility_clusterPoints.2}
        \dist(\Tilde{y}^k,C) \leq \theta \beta_k.
    \end{equation}
    Define
    \begin{equation}
        \label{eq.lem.ECM.feasibility_clusterPoints.3}
        \Tilde{x}^{k+1}\coloneqq  \frac{1}{\sigma_k} \sum_{j=0}^k \frac{\beta_j}{\eta_j} P_C(\Tilde{y}^k).
    \end{equation}
    Since $\frac{1}{\sigma_k} \sum_{j=0}^k \frac{\beta_j}{\eta_j} = 1$ by \eqref{eq.sigma_k}, we get from the convexity of $C$ that $\Tilde{x}^{k+1} \in C$. Let
    \begin{equation}
        \label{eq.lem.ECM.feasibility_clusterPoints.4}
        \Tilde{\beta}\coloneqq  \sum_{j=0}^{\infty} \beta_j^2.
    \end{equation}
    Note that $\Tilde{\beta}$ is finite by \eqref{eq.assumption_betak.1}. Then
    \begin{equation}
        \label{eq.lem.ECM.feasibility_clusterPoints.5}
        \begin{split}
            \dist(x^{k+1},C) & \leq \|x^{k+1} - \Tilde{x}^{k+1}\| = \left \| \frac{1}{\sigma_k} \left(\sum_{j=0}^k \frac{\beta_j}{\eta_j}(\Tilde{y}^j - P_C(\Tilde{y}^j)) \right) \right \| \\
            &  \leq \frac{1}{\sigma_k} \sum_{j=0}^k \frac{\beta_j}{\eta_j} \|\Tilde{y}^j - P_C(\Tilde{y}^j)\| \\
            & \leq \frac{\theta}{\sigma_k}  \sum_{j = 0}^k \frac{\beta_j^2}{\eta_j} \leq \frac{\theta}{\sigma_k}  \sum_{j = 0}^k\beta_j^2 \leq \theta \frac{\Tilde{\beta}}{\sigma_k},
        \end{split}
    \end{equation}
    using the fact that $\Tilde{x}^{k+1} \in C$ in the first inequality, (i) and \eqref{eq.lem.ECM.feasibility_clusterPoints.3} in the first equality, convexity of $\|\cdot\|$ in the second inequality, \eqref{eq.lem.ECM.feasibility_clusterPoints.2} in the third one, \eqref{eq.lem.ECM.quasi_Fejermonotonicity.1} in the fourth one, and \eqref{eq.lem.ECM.feasibility_clusterPoints.4} in the last one.

    Take $\gamma >1$ such that $\|F(\Tilde{y}^k)\| \leq \gamma$ for all $k$. Note that existence of $\gamma$ follows from (i). Then,
    \begin{equation}
        \label{eq.lem.ECM.feasibility_clusterPoints.6}
        \lim_{k \rightarrow \infty} \sigma_k = \lim_{k \rightarrow \infty} \sum_{j=0}^k \frac{\beta_j}{\eta_j} \geq \lim_{k \rightarrow \infty}\frac{1}{\gamma} \sum_{j=0}^k \beta_j = \infty,
    \end{equation}
    using the fact that $\eta_j = \max\{1,\|F(\Tilde{y})\|\} \leq \max\{1,\gamma\} \leq \gamma$ for all $j$ in the inequality, and \eqref{eq.assumption_betak.1} in the first equality. Thus, taking limits in \eqref{eq.lem.ECM.feasibility_clusterPoints.5}, we get 
    \[
        \lim_{k\rightarrow \infty} \dist(x^k,C) = 0,
    \]
    establishing (iii). 
    
    \item Item (iv) follows immediately from (iii).
    \end{listi}
\end{proof}

\begin{lemma}
\label{thm.ECM.1}
If $F$ and $C$ satisfy Assumptions \ref{ass.nonemptySOL}, \ref{a5} and \ref{assm.q}, (\textit{i.e.} 
    ${\SOL(F,C)}\neq\emptyset$ and $F$ is continuous and monotone), then all the cluster points of $\{x^k\}$ generated by Algorithm \ref{alg.ECM} are solutions of  $\VIP(F,C)$.
\end{lemma}
\begin{proof}
    For any $x \in C$ we have
    \begin{equation}
        \label{eq.thm.ECM.1.1}
        \begin{split}
            \|z^{j+1}-x\| &  = \left \|T_{\mathbf{S}_k,\mathbf{D}}\left( \Tilde{y}^j -\frac{\beta_j}{\eta_j} F(\Tilde{y}^j) \right) - T_{\mathbf{S}_k,\mathbf{D}}(x)  \right \| \leq \left \| \left( \Tilde{y}^j -\frac{\beta_j}{\eta_j} F(\Tilde{y}^j)\right) -x \right \| \\
            & = \|\Tilde{y}^j-x\|^2 + \frac{\|F(\Tilde{y}^j)\|}{\eta_j^2} \beta_j^2 - 2 \frac{\beta_j}{\eta_j} \langle F(\Tilde{y}^j), \Tilde{y}^j - x \rangle \\
            & \leq \|\Tilde{y}^j-x\|^2 +\beta_j^2 +2 \frac{\beta_j}{\eta_j} \langle F(x), x-\Tilde{y}^j \rangle,
        \end{split}
    \end{equation}
    using  nonexpansiveness of $T_{\mathbf{S}_k,\mathbf{D}}$ in the first inequality, and monotonicity of $F$ and \eqref{eq.lem.ECM.quasi_Fejermonotonicity.2} in the last inequality. Summing \eqref{eq.thm.ECM.1.1} from $0$ to $k-1$ and dividing 
    by $\sigma_{k-1}$, we obtain, in view of Lemma \ref{lem.ECM.feasibility_clusterPoints}(ii),
    \begin{equation}
        \label{eq.thm.ECM.1.2}
        \frac{(\|z^k-x\|^2 - \|z^0-x\|^2)}{\sigma_{k-1}} \leq \frac{\sum_{j=0}^{k-1}\beta_j^2}{\sigma_{k-1}} + \langle F(x),x-x^k \rangle.
    \end{equation}
    Let $\hat{x}$ be a cluster point of $\{x^k\}$. Existence of $\hat{x}$ is guaranteed by Lemma \ref{lem.ECM.feasibility_clusterPoints}(i). Note that $\hat{x}\in C$ by Lemma \ref{lem.ECM.feasibility_clusterPoints}(iv).

    By \eqref{eq.assumption_betak.1}, \eqref{eq.assumption_betak.2}, \eqref{eq.lem.ECM.feasibility_clusterPoints.6} and boundedness of $\{z^k\}$, taking limits in \eqref{eq.thm.ECM.1.2} with $k \rightarrow \infty$, we obtain that 
    \[
        \langle F(x), x-\hat{x} \rangle \geq 0,
    \]
    for all $x\in C$. Using Proposition \ref{prop.Minty}, we get that $\hat{x}\in \SOL(F,C)$. Therefore, all the cluster points of $\{x^k\}$ solve  $\VIP(F,C)$.
\end{proof}
 We complete next the convergence analysis of the algorithm.
\begin{theorem}
\label{thm.ECM.2} 
If $F$ and $C$ satisfy Assumptions \ref{a5} and \ref{assm.q}, (\textit{i.e.} $F$ is continuous and monotone), then, either $\SOL(F,C) \neq \emptyset$ and $\{x^k\}$ converges to $\bar{x}$, where $\bar{x} = \lim_{k\rightarrow \infty} P_{\SOL(F,C)}(\Tilde{y}^k)$ or $\SOL(F,C) = \emptyset$ and $\lim_{k \rightarrow \infty} \|x^k\| = \infty$.
\end{theorem}
\begin{proof}
    Assume that $\SOL(F,C) \neq \emptyset$ and define $u^k =P_{\SOL(F,C)}(\Tilde{y}^k)$. Note that $u^k$ exists because $\SOL(F,C)$ is nonempty, closed, and convex by Proposition \ref{prop.SOL}. By Lemma \ref{lem.ECM.quasi_Fejermonotonicity}, $\{\Tilde{y}^k\}$ is quasi-Fej\'er convergent to $\SOL(F,C)$ since $\SOL(F,C) \subset C$. Therefore, it follows from Proposition \ref{prop.convergence_projected_sequence}(ii) that $\{P_{\SOL(F,C)}(\Tilde{y}^k)\}$ is convergent. Let
    \begin{equation}
        \label{eq.thm.ECM.2.1}
        \bar{x} = \lim_{k\rightarrow \infty} P_{\SOL(F,C)}(\Tilde{y}^k) = \lim_{k\rightarrow \infty} u^k.
    \end{equation}
    By Lemma \ref{lem.ECM.quasi_Fejermonotonicity} and Lemma \ref{thm.ECM.1}, $\{x^k\}$ is bounded and each of its cluster points belongs to $\SOL(F,C)$. Suppose that $\hat{x} \in \SOL(F,C)$ is a cluster point of $\{x^k\}$ and $\{x^{i_k}\} \rightarrow \hat{x}$. It suffices to show that $\bar{x} = \hat{x}$.

    By Proposition \ref{prop.ProjectionTheorem}(i) we have that $\langle \hat{x}-u^j,\Tilde{y}^j-u^j \rangle \leq 0$ for all $j$. Let $\epsilon = \sup_{0\leq j \leq \infty} \|\Tilde{y}^j-u^j\|$. By Lemma \ref{lem.ECM.feasibility_clusterPoints}(i), $\epsilon<\infty$. Then,
    \begin{equation}
        \label{eq.thm.ECM.2.2}
        \langle \hat{x}-\bar{x}, \Tilde{y}^j-u^j \rangle \leq \langle u^j-\bar{x},\Tilde{y}^j-u^j \rangle \leq \epsilon \|u^j-\bar{x}\|,
    \end{equation}
    for all $j$. Multiplying \eqref{eq.thm.ECM.2.2} by $\frac{\beta_j}{\eta_j \sigma_{k-1}}$ and summing from $0$ to $k-1$, we get from Lemma \ref{lem.ECM.feasibility_clusterPoints}(ii)
    \begin{equation}
        \label{eq.thm.ECM.2.3}
        \left \langle \hat{x}-\bar{x}, x^k - \frac{1}{\sigma_{k-1}} \sum_{j=0}^{k-1} \frac{\beta_j}{\eta_j} u^j \right \rangle \leq \frac{\epsilon}{\sigma_{k-1}} \sum_{j=0}^{k-1} \frac{\beta_j}{\eta_j} \|u^j-\bar{x}\|.
    \end{equation}
    Define $\zeta_{k,j} \coloneqq  \frac{1}{\sigma_k} \frac{\beta_j}{\eta_j} (k \geq 0,~0\leq j \leq k)$. In view of \eqref{eq.lem.ECM.feasibility_clusterPoints.6}, $\lim_{k \rightarrow \infty} \zeta_{k,j} = 0$ for all $j$. By \eqref{eq.thm.ECM.2.1},
    \begin{equation}
        \sum_{j=0}^k \zeta_{k,j} = 1,
    \end{equation}
    for all $k$. Then, using \eqref{eq.thm.ECM.2.1} and Proposition \ref{prop.elemntary} with
    \begin{equation}
        \label{eq.thm.ECM.2.4}
        w^k = \sum_{j=0}^k \zeta_{k,j} u^j = \sum_{j=0}^k \frac{\beta_j}{\eta_j} u^j,
    \end{equation}
    we obtain that
    \begin{equation}
        \label{eq.thm.ECM.2.5}
        \lim_{k \rightarrow \infty} \frac{1}{\sigma_{k-1}} \sum_{j=0}^k \frac{\beta_j}{\eta_j} u^j = \bar{x}, 
    \end{equation}
    and hence
    \begin{equation}
        \label{eq.thm.ECM.2.6} 
        \lim_{k \rightarrow \infty} \frac{1}{\sigma_{k-1}} \sum_{j=0}^k \frac{\beta_j}{\eta_j} \|u^j-\bar{x}\| = 0,
    \end{equation}
    using the fact that $\frac{1}{\sigma_{k-1}} \sum_{j=0}^k \frac{\beta_j}{\eta_j} =1$.

    From \eqref{eq.thm.ECM.2.5} and \eqref{eq.thm.ECM.2.6}, since $\lim_{k \rightarrow \infty}x^{i_j} =\bar{x}$, taking limits in \eqref{eq.thm.ECM.2.4} with $k\rightarrow \infty$ along the subsequence with subindices $\{i_k\}$, we get that
    \[
        \langle\hat{x}-\bar{x}, \hat{x} -\bar{x} \rangle \leq 0,
    \]
    implying that $\hat{x} = \bar{x}$.

If $\SOL(F,C)$ is empty then by Lemma \ref{thm.ECM.1} no subsequence of $\{x^k\}$ can be bounded, and hence $\lim_{k \rightarrow \infty} \|x^k\| = \infty$.
\end{proof}

We remark that the sequence $\{x^k\}$ is a weighted average of the $k$ first elements of the sequence $\{\Tilde y^k\}$. It is easy to check that if $\{\Tilde y^k\}$ converges then so does 
$\{x^k\}$, but the converse statement does not hold. Indeed, if we consider $\VIP(F,C)$ in $\re^2$, 
taking $F$ as the $\pi/2$-rotation defined in Subsection \ref{subs} and $C$ as the unit disk, then the only solution of $\VIP(F,C)$ is $0$, but the sequence $\{\Tilde y^k\}$ does not converge; rather, it accumulates on the boundary of $C$. The averaged sequence $\{x^k\}$, on the other hand, does converge to $0$. In all our numerical experiments in Section 5, we monitor both $\{\Tilde y^k\}$ and $\{x^k\}$. Whenever $\{\Tilde y^k\}$ converges (which happens almost always), we report the results related to $\{\Tilde y^k\}$ instead of $\{x^k\}$. Obviously, in these cases $\{\Tilde y^k\}$ converges faster than $\{x^k\}$, whose convergence is slowed down by the effect of the initial elements of $\{\Tilde y^k\}$ in the computation of $x^k$. 



\section{Numerical experiments}

In this section, we present numerical experiments to evaluate the performance of our proposed Algorithms~\ref{alg.ACRM_paramonotone} and~\ref{alg.ECM}, which we refer to as \AlgOne\ and \AlgTwo, respectively.

\subsection{Setup and methodology}
We conduct experiments comparing our proposed algorithms with established methods for solving variational inequality problems (VIPs).
To put our methods to the test, beyond Algorithm~2 in~\cite{Bello:2010} (referred to as \AlgBIOne), Algorithm~A in~\cite{Bello:2012} (denoted by \AlgBITwo), we selected the classical Extragradient method~\cite{kor} (called here \EGM), and Algorithm~4.2 of~\cite{Malitsky:2015} (abbreviated as \AlgMalAdap). We recall that \AlgOne\ and \AlgTwo\ are circumcenter-based methods intended to accelerate \AlgBIOne\ and \AlgBITwo, respectively. The methods \AlgOne, \AlgTwo, \AlgBIOne, and \AlgBITwo\ employ \emph{approximate projections} (based on separating halfspaces), while \EGM\ and \AlgMalAdap\ use \emph{exact projections}.

We consider the feasible set $C$ defined as the intersection of ellipsoids:
\begin{equation}
\label{eq.ellipsoid}
C = \bigcap_{i=1}^m C_i,\quad C_i \coloneqq \left\{x \in \mathds{R}^n : g_i(x) \leq 0 \right\}, \text{ where } g_i(x) =  \scal{x}{A_i x} + 2\scal{b^i}{x} - \alpha,
\end{equation}
where each $A_i \in \mathds{R}^{n \times n}$ is a positive definite matrix, $b^i \in \mathds{R}^n$, and $\alpha \in \mathds{R}$. These ellipsoids are generated following the approach in~\cite{Behling:ScCRM}*{Section~5}, ensuring the existence of a Slater point, which is required for the inner loop of both \AlgBITwo\ and \AlgTwo.

For computing the circumcenter step in~\eqref{equ270}, we use the gradient of $g_i$, given by $\nabla g_i(x) = 2A_i x + 2b^i$. However, computing the exact projection onto $C = \bigcap_{i=1}^m C_i$, as required by \EGM\ and \AlgMalAdap, is nontrivial. We employ Dykstra's algorithm~\cite{Dykstra:1983} for computing the projection onto $C$, where each subprojection onto $C_i$ is computed via ADMM as described in~\cite{Jia:2017}*{Algorithm~6}.

Importantly, \AlgOne\ and \AlgTwo\ share the same codebase as \AlgBIOne\ and \AlgBITwo, respectively, differing only in the projection step: \AlgOne\ and \AlgTwo\ replace the approximate projection with the circumcenter acceleration. This design ensures a fair comparison, isolating the effect of the circumcenter acceleration.

Given $x \notin C$, we define
\begin{equation}\label{eq.ghat}
\hat{g}(x) := \max_{i \in \{1,\ldots,m\}} g_i(x),
\end{equation}
and let $k$ be any index attaining the maximum. The halfspace
\begin{equation}
\label{eq.hatH}
\hat{H}(x) := \left\{z \in \mathds{R}^n : g_k(x) + \langle \nabla g_k(x), z - x \rangle \leq 0 \right\},
\end{equation}
separates $x$ from $C$, and the projection onto $\hat{H}(x)$ admits a closed form. The existence of a Slater point, guaranteed by the ellipsoid generation procedure, is required for the inner loops of \AlgTwo\ and \AlgBITwo.

The experiments are organized into three scenarios according to problem size:
\begin{lista}
    \item \textbf{Scenario A (Small):} $n \in \{5, 10\}$ and $m \in \{2, 5\}$. All six methods are tested.
    
    \item \textbf{Scenario B (Medium):} $n \in \{50, 100\}$ and $m \in \{5, 8\}$. Only approximate projection methods (\AlgOne, \AlgTwo, \AlgBIOne, \AlgBITwo) are tested.
    
    \item \textbf{Scenario C (Large):} $n \in \{100, 200, 500\}$ and $m \in \{20, 30, 50\}$. Only approximate projection methods are tested.
\end{lista}
This structure allows us to assess both the accuracy trade-off between exact and approximate projections (Scenario~A) and the scalability of approximate projection methods as well as checking the benefits of circumcenter-type methods (Scenarios~B and~C).

All experiments were implemented in Julia~1.12~\cite{bezansonJuliaFreshApproach2017} and executed on a MacBook Air M4 with 24 GB RAM and 10-core CPU. The source code is available at \url{https://github.com/lrsantos11/CRM_VIP.jl}. The following conditions were used:
\begin{listi}
    \item The maximum number of iterations is $\num{100000}$ for Scenario~A and $\num{300000}$ for Scenarios~B and~C.
    
    \item The stopping criteria vary by method:
    \begin{itemize} 
        \item For \AlgOne\ and \AlgBIOne:
        \begin{equation}
            \label{eq.stop.alg1}
            \frac{\|x^{k+1} - x^k\|}{\max\{\|x^k\|, 1\}} \leq \varepsilon.
        \end{equation}
        \item For \AlgTwo\ and \AlgBITwo, the algorithm stops when either the above criterion holds for the ergodic sequence $\{x^k\}$, or
        \begin{equation}
            \label{eq.stop.alg2}
            \|z^{k} - \tilde{y}^k\| \leq \varepsilon.
        \end{equation}
        \item For \EGM:
        \begin{equation}
            \label{eq.stop.egm}
            \|x^k - y^k\| \leq \varepsilon.
        \end{equation}
        \item For \AlgMalAdap, we use the residual function from~\cite{Malitsky:2015}:
        \begin{equation}
            \label{eq.stop.mal}
            r(x^k, y^k) = \|y^k - P_C(y^k - \lambda F(y^k))\| + \|x^k - y^k\| \leq \varepsilon.
        \end{equation}
    \end{itemize}
    We set $\varepsilon = 10^{-6}$ for Scenarios~A and~B, and $\varepsilon = 10^{-5}$ for Scenario~C.
    
    
    \item We set $\beta_k = 1/k^{0.9}$ for \AlgOne, \AlgTwo, \AlgBIOne, and \AlgBITwo, which satisfies~\eqref{eq.assumption_betak.1}--\eqref{eq.assumption_betak.2} and the analogous conditions in~\cites{Bello:2010,Bello:2012}. For \AlgMalAdap, the step size is adapted at each iteration as described in~\cite{Malitsky:2015}*{Alg.~4.2}, with initial step size $\lambda_0 = 0.5/L$, where $L$ is an estimate of the Lipschitz constant of $F$.
    
    \item For each configuration $(n, m)$, we generate $10$ random instances for Scenarios~A and~B, and $5$ instances for Scenario~C.
    
    \item In all scenarios we run the experiments in \num{10} threads to parallelize the execution of different instances. When the instance had more than $8$ sets, the computation of the circumcenter step was also parallelized Julia package \texttt{ThreadsX.jl}.
    
    \item Wall-clock times are measured using the \texttt{BenchmarkTools.jl} package~\cite{chenRobustBenchmarkingNoisy2016a}, which reports the minimum time over multiple executions, providing reliable and reproducible timing measurements.
\end{listi}

\subsection{Test problems}\label{sec:test_problems}

We evaluate the algorithms on three families of operators $F$ for $\VIP(F,C)$, designed to test different theoretical properties: gradient operators (paramonotone), non-gradient paramonotone operators, and monotone but non-paramonotone operators.

\begin{example}[Gradient of convex function -- paramonotone]\label{Example_1}
We consider operators $F$ that are gradients of convex functions. In this case, solving $\VIP(F,C)$ is equivalent to minimizing a convex function over $C$. The operator is defined as
\begin{equation}
    \label{eq.ex1.operator}
    F(x) = Ax + G(x) + c,
\end{equation}
where $G(x)_i = b_i x_i^3$ with $b \in \mathds{R}^n$, $b \geq 0$, and $c \in \mathds{R}^n$, $c \neq 0$. The matrix $A \in \mathds{R}^{n \times n}$ is generated randomly as symmetric positive semidefinite (via $A = MM^\top$ for a random matrix $M$). Note that $F$ is the gradient of the convex function 
\[
f(x) = \frac{1}{2}\scal{x}{Ax} + \scal{c}{x} + \frac{1}{4}\sum_{i=1}^n b_i x_i^4.
\]
Since gradients of convex functions are paramonotone~\cite{Censor:1998}, all algorithms are expected to converge.

\end{example}

\begin{example}[Paramonotone, non-gradient operator]\label{Example_2}
We now consider operators that are paramonotone but not gradients of convex functions. The operator is defined as
\begin{equation}
    \label{eq.ex2.operator}
    F(x) = Ax + c,
\end{equation}
where $A \in \mathds{R}^{n \times n}$ has the block-diagonal structure
\begin{equation}
\label{eq.ex2.matrix}
    A = \begin{bmatrix}
        A_1 & 0  \\
        0 & A_2 
    \end{bmatrix},
\end{equation}
with $A_1 \in \mathds{R}^{n_1 \times n_1}$ symmetric positive semidefinite and $A_2 \in \mathds{R}^{n_2 \times n_2}$ nonsymmetric positive semidefinite (where $n_1 + n_2 = n$). We construct $A_2$ following~\cite{Malitsky:2015} as $A_2 = M_2 M_2^\top + B + D$, where $B$ is skew-symmetric and $D$ is diagonal with positive entries. Since $A$ is not symmetric, $F$ is not a gradient. However, $F$ is paramonotone because $\rank(A + A^\top) = \rank(A)$.
\end{example}

\begin{example}[Monotone, non-paramonotone operator]\label{Example_3}
Finally, we consider monotone operators that are not paramonotone. The operator $F$ is still defined by~\eqref{eq.ex2.operator} with the block structure~\eqref{eq.ex2.matrix}, but now $A_2 \in \mathds{R}^{n_2 \times n_2}$ is \emph{skew-symmetric}, i.e.,
\begin{equation}
    \label{eq.ex3.skew}
    A_2^\top = -A_2.
\end{equation}
The operator $F$ is monotone since $\scal{Ax}{x} = \scal{A_1 x_1}{x_1} + \scal{A_2 x_2}{x_2} = \scal{A_1 x_1}{x_1} \geq 0$ (as $\scal{A_2 x_2}{x_2} = 0$ for skew-symmetric $A_2$). However, $F$ is not paramonotone because $\rank(A + A^\top) < \rank(A)$ when $A_2 \neq 0$. In this example, we expect that algorithms requiring paramonotonicity (\AlgOne\ and \AlgBIOne) may fail to converge in some instances, while \AlgBITwo, \AlgTwo, \EGM,  and \AlgMalAdap\ (which only require monotonicity) should converge.
\end{example}

\subsection{Discussion on the numerical results}

For each example, we present results in tables, in Appendix \ref{appendix:numerical_results}, comparing iteration counts and wall-clock time, as well as speedup comparison with the fastest method. Scenarios B and C were combined into a single table for brevity, as the same methods were compared under both scenarios. All the tables report median values over 10 random instances for each configuration of dimension $n$ and number of ellipsoids $m$.

In order to summarize all experiments, in Figures \ref{fig:perf_Ex1}, \ref{fig:perf_Ex2}, and \ref{fig:perf_Ex3}, we present performance profiles~\cite{Dolan:2002} for each of the three examples, illustrating the comparative performance of the algorithms across all tested instances. Each figure contains four subplots: the first two correspond to Scenario A (iterations and wall-clock time), while the last two correspond to Scenarios B and C (iterations and wall-clock time). These profiles provide a visual representation of how each algorithm performs relative to the others across different problem instances.

\begin{figure}[htpb]
     \centering
     \begin{subfigure}[b]{0.48\textwidth}
         \centering
         \includegraphics[width=\textwidth]{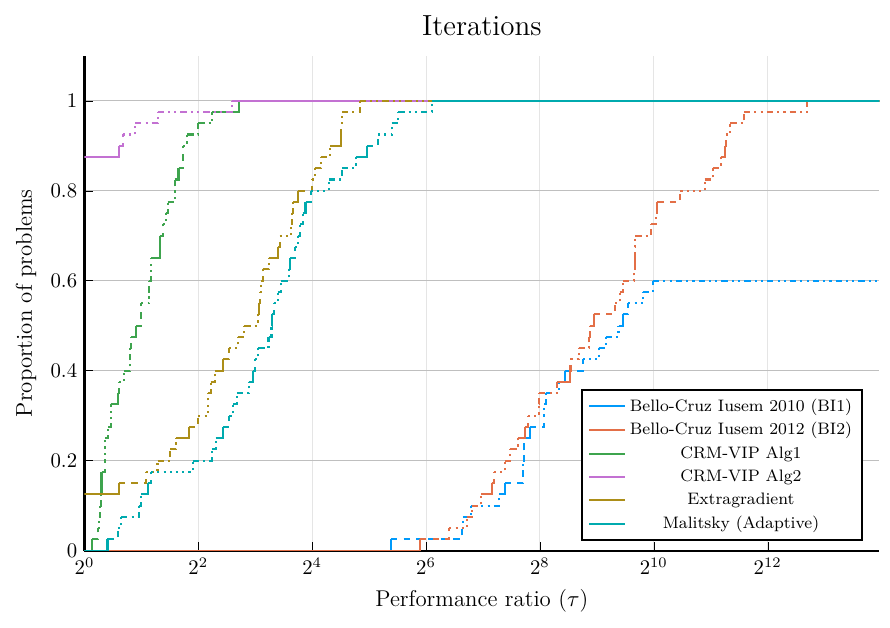}
         \caption{Scenario A: Iterations}
         \label{fig:prof_51_A_iter}
     \end{subfigure}
     \hfill
     \begin{subfigure}[b]{0.48\textwidth}
         \centering
         \includegraphics[width=\textwidth]{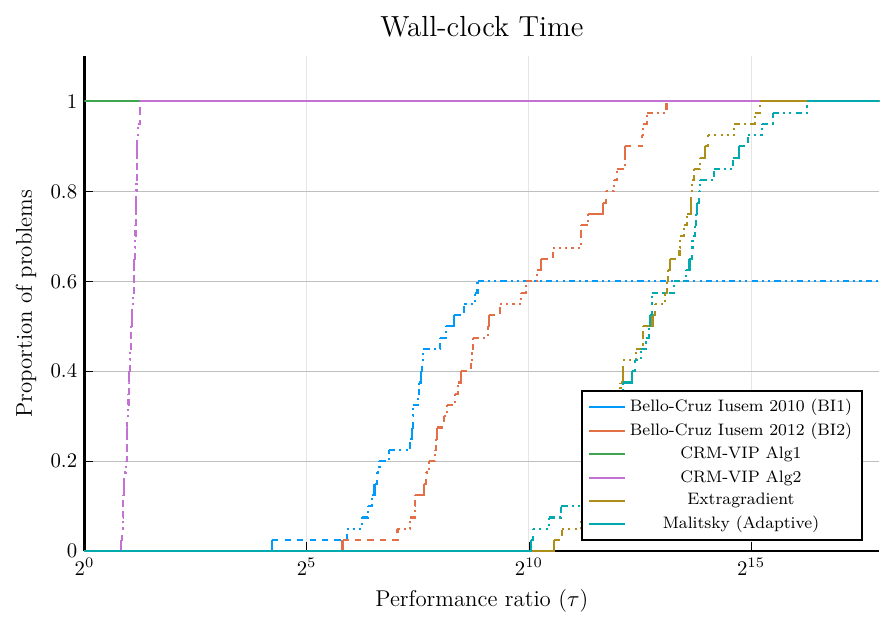}
         \caption{Scenario A: Wall-clock time}
         \label{fig:prof_51_A_time}
     \end{subfigure}

     \vspace{0.5cm}

     \begin{subfigure}[b]{0.48\textwidth}
         \centering
         \includegraphics[width=\textwidth]{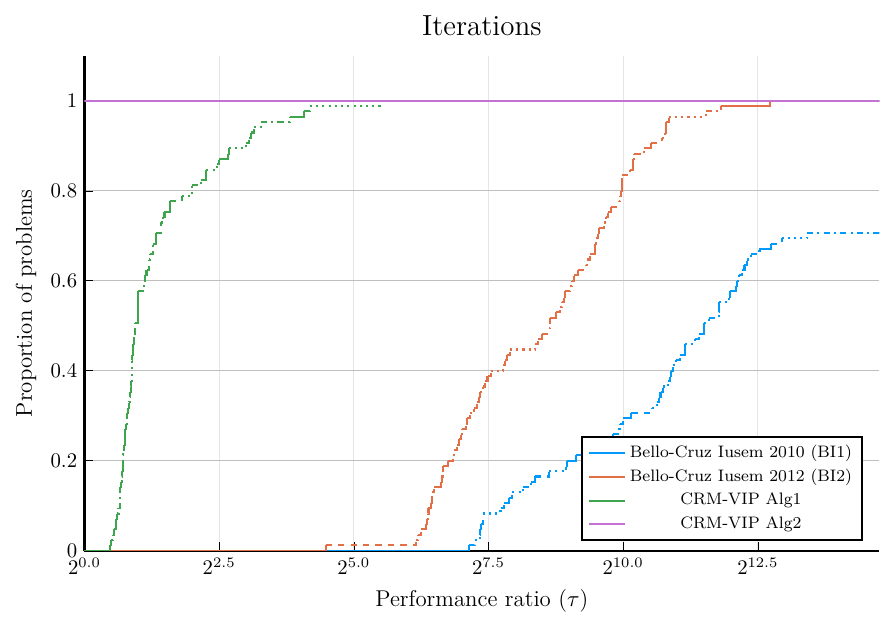}
         \caption{Scenarios B and C: Iterations}
         \label{fig:prof_51_BC_iter}
     \end{subfigure}
     \hfill
     \begin{subfigure}[b]{0.48\textwidth}
         \centering
         \includegraphics[width=\textwidth]{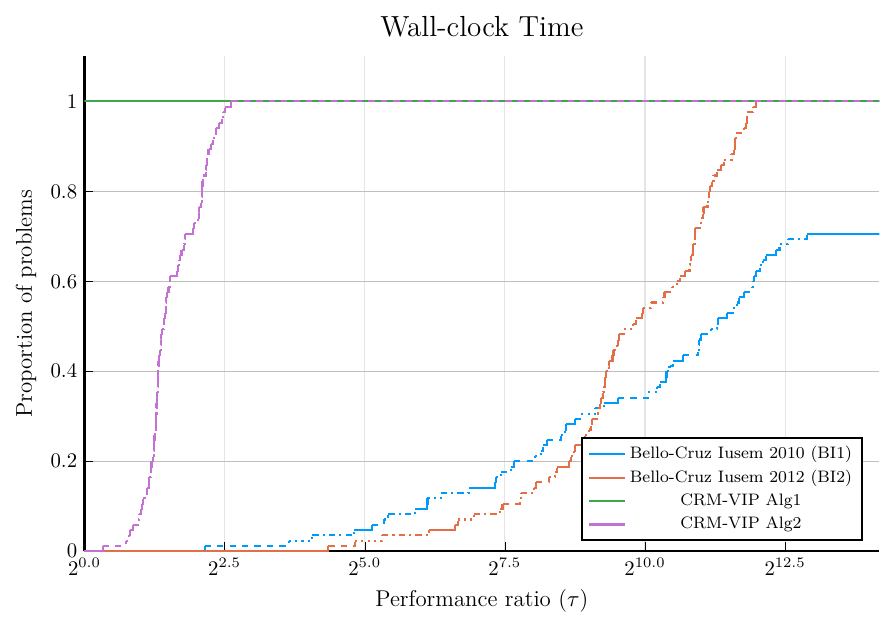}
         \caption{Scenarios B and C: Wall-clock time}
         \label{fig:prof_51_BC_time}
     \end{subfigure}
     \caption{Performance profiles for Example \ref{Example_1} (gradient operator).}
     \label{fig:perf_Ex1}
\end{figure}

\begin{figure}[htpb]
     \centering
     \begin{subfigure}[b]{0.48\textwidth}
         \centering
         \includegraphics[width=\textwidth]{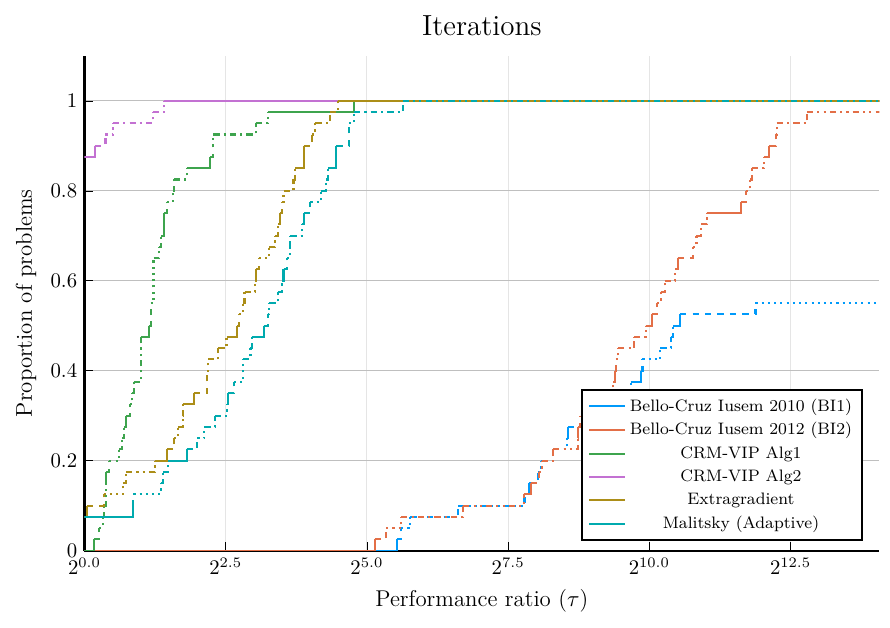}
         \caption{Scenario A: Iterations}
         \label{fig:prof_52_A_iter}
     \end{subfigure}
     \hfill
     \begin{subfigure}[b]{0.48\textwidth}
         \centering
         \includegraphics[width=\textwidth]{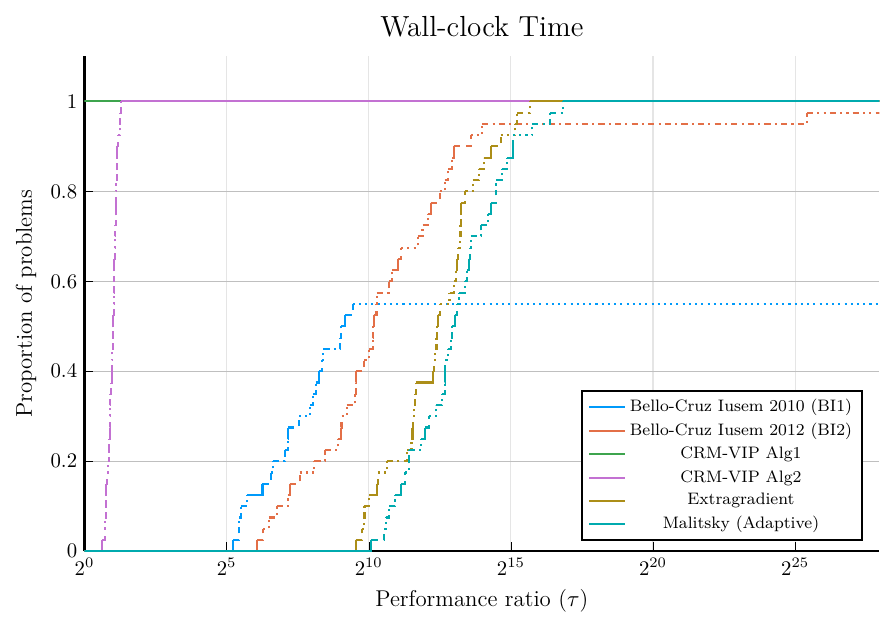}
         \caption{Scenario A: Wall-clock time}
         \label{fig:prof_52_A_time}
     \end{subfigure}

     \vspace{0.5cm}

     \begin{subfigure}[b]{0.48\textwidth}
         \centering
         \includegraphics[width=\textwidth]{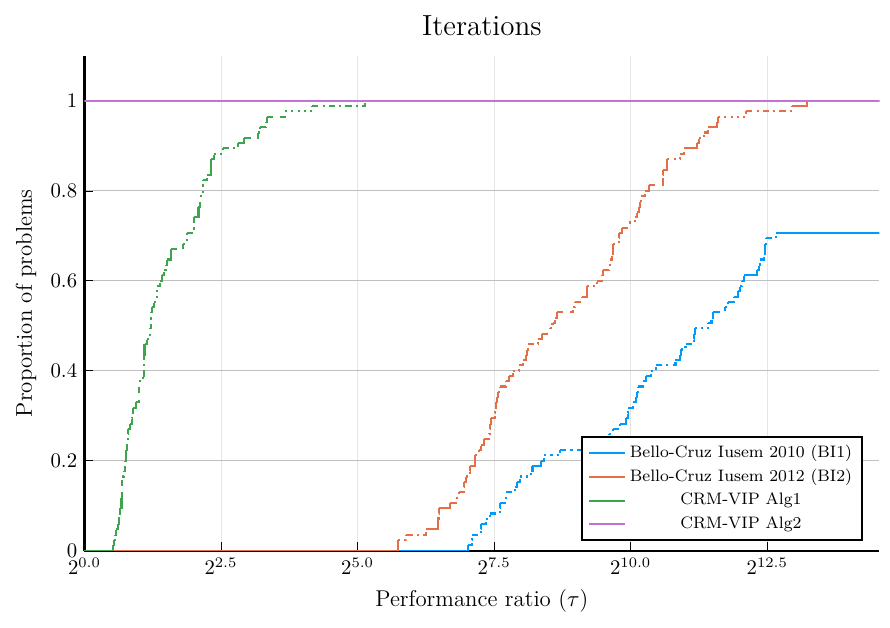}
         \caption{Scenarios B and C: Iterations}
         \label{fig:prof_52_BC_iter}
     \end{subfigure}
     \hfill
     \begin{subfigure}[b]{0.48\textwidth}
         \centering
         \includegraphics[width=\textwidth]{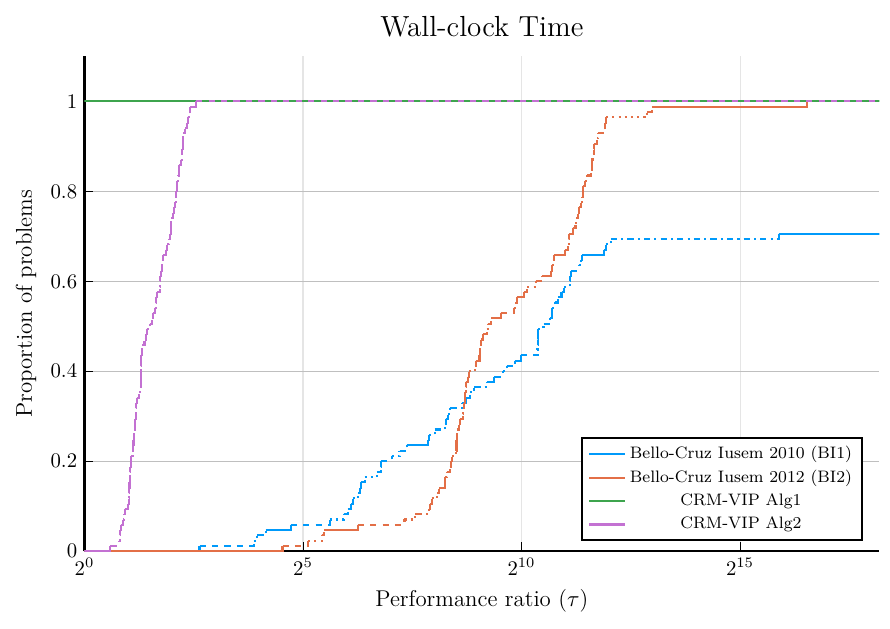}
         \caption{Scenarios B and C: Wall-clock time}
         \label{fig:prof_52_BC_time}
     \end{subfigure}
     \caption{Performance profiles for Example \ref{Example_2} (paramonotone, non-gradient operator).}
     \label{fig:perf_Ex2}
\end{figure}

\begin{figure}[htpb]
     \centering
     \begin{subfigure}[b]{0.48\textwidth}
         \centering
         \includegraphics[width=\textwidth]{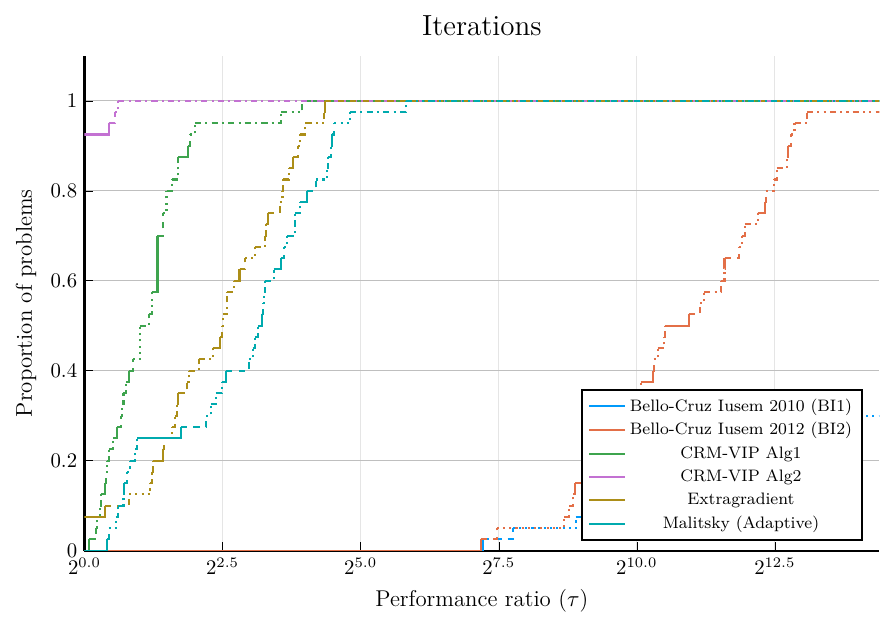}
         \caption{Scenario A: Iterations}
         \label{fig:prof_53_A_iter}
     \end{subfigure}
     \hfill
     \begin{subfigure}[b]{0.48\textwidth}
         \centering
         \includegraphics[width=\textwidth]{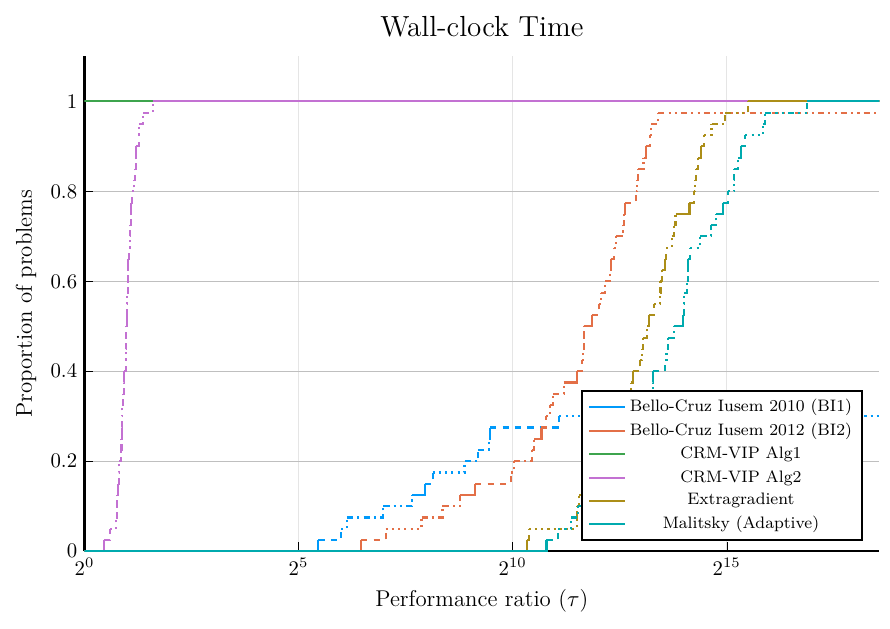}
         \caption{Scenario A: Wall-clock time}
         \label{fig:prof_53_A_time}
     \end{subfigure}

     \vspace{0.5cm}

     \begin{subfigure}[b]{0.48\textwidth}
         \centering
         \includegraphics[width=\textwidth]{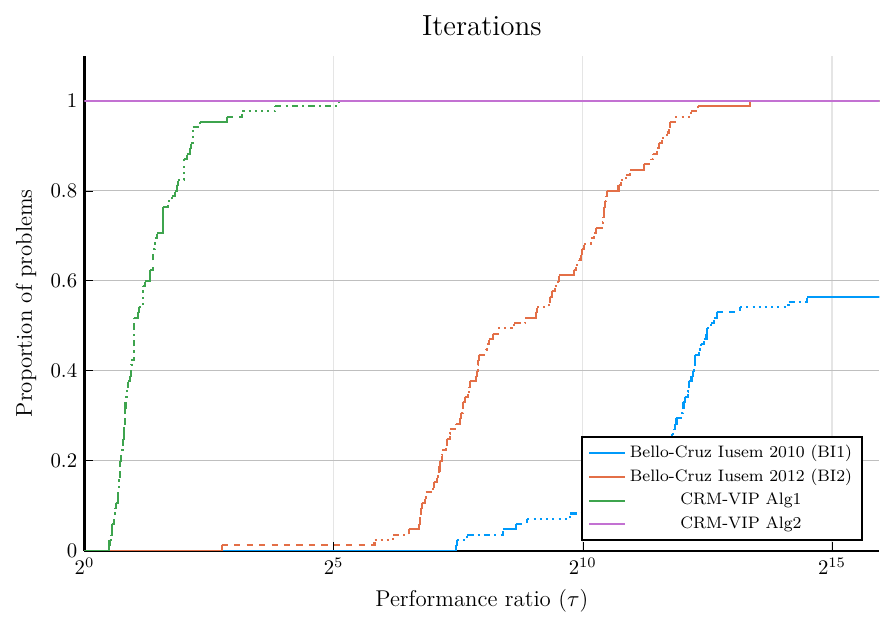}
         \caption{Scenarios B and C: Iterations}
         \label{fig:prof_53_BC_iter}
     \end{subfigure}
     \hfill
     \begin{subfigure}[b]{0.48\textwidth}
         \centering
         \includegraphics[width=\textwidth]{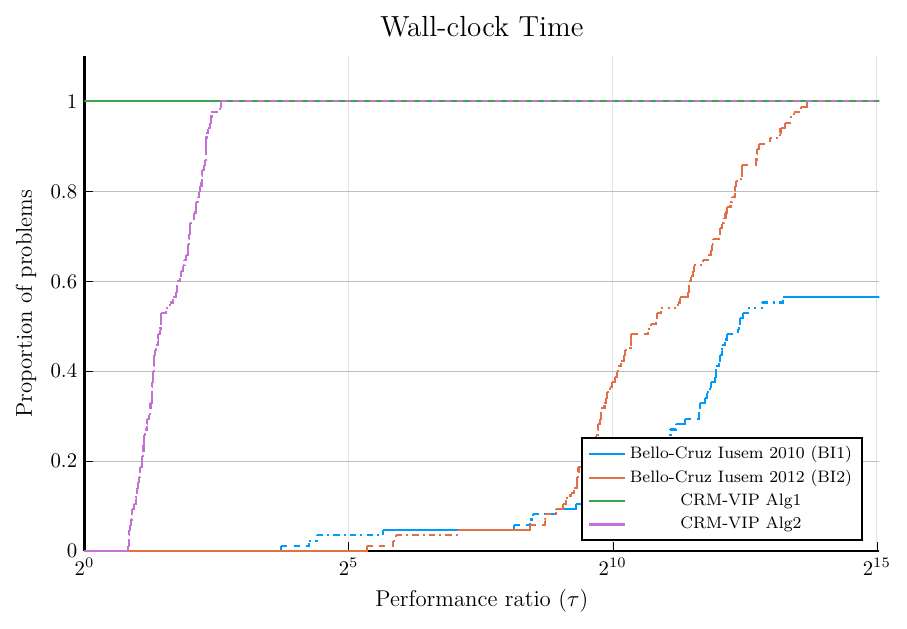}
         \caption{Scenarios B and C: Wall-clock time}
         \label{fig:prof_53_BC_time}
     \end{subfigure}
     \caption{Performance profiles for Example \ref{Example_3} (monotone, non-paramonotone operator).}
     \label{fig:perf_Ex3}
\end{figure}

Based on the results of the three types of numerical experiments reported in Appendix \ref{appendix:numerical_results} and summarized in the performance profiles (Figures \ref{fig:perf_Ex1}--\ref{fig:perf_Ex3}), we draw the following observations:

\begin{enumerate}
    \item \textbf{Comparison with Exact Projection Methods (Scenario A):} 
    Tables \ref{tab:results_Example_1_A}, \ref{tab:results_Example_2_A}, and \ref{tab:results_Example_3_A} demonstrate a stark contrast between methods using exact projections (\EGM\ and \AlgMalAdap) and the proposed approximate methods. \AlgOne\ and \AlgTwo\ are consistently three to four orders of magnitude faster than \EGM\ and \AlgMalAdap. For instance, in Table \ref{tab:speedup_Example_1_A} ($n=10, m=2$), \AlgOne\ is over \num{10000} times faster than \EGM. This massive speedup is attributed to the computational cost of projecting onto the intersection of ellipsoids using Dykstra's algorithm, whereas our methods utilize computationally cheap separating halfspaces. This confirms that for complex feasible sets defined by intersections, approximate projection strategies are computationally superior.

    \item \textbf{Impact of Circumcenter Acceleration (Scenarios A, B, and C):} 
    The benefit of the circumcenter acceleration is evident when comparing the proposed methods to their non-accelerated counterparts. 
    \begin{itemize}
        \item \textbf{\AlgOne\ vs.~\AlgBIOne:} In Scenarios B and C (Tables \ref{tab:results_Example_1_B_C}--\ref{tab:results_Example_3_B_C}), \AlgOne\ reduces the iteration count from tens of thousands (often hitting the max-iteration cap) to merely 15--30 iterations. The speedup factors range from roughly $200\times$ to over $3000\times$.
        \item \textbf{\AlgTwo\ vs.~\AlgBITwo:} Similarly, \AlgTwo\ exhibits a dramatic improvement over \AlgBITwo, reducing iteration counts from thousands to single digits (often fewer than 10 iterations). The speedup relative to \AlgBITwo\ frequently exceeds $1000\times$.
    \end{itemize}
    This clearly establishes the superiority of the circumcenter approach, which effectively mitigates the ``zig-zagging'' behavior often observed in standard separating hyperplane methods.

    \item \textbf{Robustness in Non-Paramonotone Settings (Example \ref{Example_3}):} 
    In Example \ref{Example_3}, where the operator is monotone but not paramonotone, the theoretical limitations of \AlgBIOne\ become apparent. As shown in Table \ref{tab:results_Example_3_B_C}, \AlgBIOne\ fails to converge in some instances (e.g., $n=100, m=5$) or requires an excessive number of iterations. Interestingly, its accelerated version,  \AlgOne, continues to converge efficiently, suggesting that the circumcenter step may induce a stabilizing effect even outside the strict theoretical scope of paramonotonicity. However, \AlgTwo, which is theoretically designed for general monotone operators, proves to be the most robust choice, solving all instances in Example \ref{Example_3} with the lowest iteration counts (typically under 10).

    \item \textbf{Efficiency vs. Accuracy Trade-off:} 
    While \EGM\ and \AlgMalAdap\ theoretically provide iterates that lie exactly within the feasible set $C$, the computational cost is prohibitive for large $m$. Both \AlgOne\ and \AlgTwo\ generate infeasible iterates that approach $C$ asymptotically. However, as shown by the stopping criteria results, the final iterates of the proposed algorithms achieve high accuracy in both feasibility and optimality (solution quality) in a fraction of the time. Between the two proposed methods, \AlgOne\ is generally faster per iteration (as seen in the speedup tables where \AlgOne\ is typically $2\times$ to $5\times$ faster than \AlgTwo), but \AlgTwo\ requires fewer iterations and offers stronger theoretical guarantees for general monotone operators.

    \item \textbf{Global Performance Assessment:} 
    The performance profiles in Figures \ref{fig:perf_Ex1}, \ref{fig:perf_Ex2}, and \ref{fig:perf_Ex3} visually summarize these findings. In all scenarios, the curves for \AlgOne\ and \AlgTwo\ rise rapidly to 1.0, indicating that they are the most efficient methods on the vast majority of test problems. Specifically, \AlgOne\ dominates in terms of wall-clock time due to the simplicity of its iteration, while \AlgTwo\ often dominates in terms of iteration count.
\end{enumerate}

In summary, \AlgOne\ and \AlgTwo\ demonstrate superior performance over \AlgBIOne, \AlgBITwo, \AlgMalAdap, and \EGM. We remark that \AlgOne\ is the fastest method overall, making it the preferred choice for paramonotone problems or gradients of convex functions. On the other hand, \AlgTwo\ offers a compelling alternative for general monotone problems, providing extreme iteration reduction and robustness at a slightly higher cost per iteration than \AlgOne.

\section{Concluding remarks}

In this work, we have proposed two distinct circumcenter-type projection methods for solving variational inequality problems involving paramonotone and monotone operators, respectively. One method is equipped with global convergence guarantees, while the other admits ergodic convergence. The numerical experiments provide compelling evidence of the superiority of the proposed approach. By avoiding computationally expensive exact projections in favor of circumcenter steps onto separating halfspaces, our algorithms achieve a dramatic reduction in computational time—outperforming classical exact projection methods by orders of magnitude and providing robust acceleration over existing approximation schemes, particularly when the intersection \( \bigcap_{i=1}^m C_i \) is complex.

Future research will focus on extending these efficient schemes to broader classes of operators, such as quasi-monotone, or to stochastic settings, as well as investigating their applications in convex optimization and equilibrium problems. Additionally, refining the theoretical analysis of convergence rates and strengthening the ergodic convergence of Algorithm \ref{alg.ECM} to nonergodic global convergence remain important avenues for future investigation.





\bibliography{refs_mcom}

\appendix

\section{Tables of numerical results}
\label{appendix:numerical_results}

\begin{table}[htbp]
\centering
\caption{Median of iterations and wall-clock time (s) -- Example \ref{Example_1}.A}
\label{tab:results_Example_1_A}
\resizebox{\textwidth}{!}{%
\begin{tabular}{rrcccccccccccc}
\toprule
$n$ & $m$ & \multicolumn{2}{c}{\EGM} & \multicolumn{2}{c}{\AlgMalAdap} & \multicolumn{2}{c}{\AlgBIOne} & \multicolumn{2}{c}{\AlgOne} & \multicolumn{2}{c}{\AlgBITwo} & \multicolumn{2}{c}{\AlgTwo} \\
 &  & Iter & Time (s) & Iter & Time (s) & Iter & Time (s) & Iter & Time (s) & Iter & Time (s) & Iter & Time (s) \\
\cmidrule(lr){3-4} \cmidrule(lr){5-6} \cmidrule(lr){7-8} \cmidrule(lr){9-10} \cmidrule(lr){11-12} \cmidrule(lr){13-14} 
5 & 2 & 32 & \num{7.21e-02} & 50 & \num{6.11e-02} & 2065 & \num{2.01e-03} & 18 & \textbf{\num{1.43e-05}} & 2559 & \num{5.14e-03} & \textbf{10} & \num{3.06e-05} \\
5 & 5 & 37 & \num{1.34e-01} & 62 & \num{1.64e-01} & 2805 & \num{3.80e-03} & 10 & \textbf{\num{1.58e-05}} & 5138 & \num{2.39e-02} & \textbf{4} & \num{2.91e-05} \\
10 & 2 & 66 & \num{2.38e-01} & 110 & \num{6.07e-01} & 3403 & \num{4.14e-03} & 20 & \textbf{\num{2.21e-05}} & 10538 & \num{5.94e-02} & \textbf{16} & \num{4.82e-05} \\
10 & 5 & 71 & \num{1.61e-01} & 69 & \num{1.35e-01} & 2094 & \num{3.56e-03} & 16 & \textbf{\num{2.79e-05}} & 2802 & \num{1.07e-02} & \textbf{12} & \num{6.28e-05} \\
\bottomrule
\end{tabular}%
}
\end{table}

\begin{table}[htbp]
\centering
\caption{Speedup relative to \AlgOne -- Example \ref{Example_1}.A}
\label{tab:speedup_Example_1_A}
{%
\begin{tabular}{rrccccc}
\toprule
$n$ & $m$ & \EGM & \AlgMalAdap & \AlgBIOne & \AlgBITwo & \AlgTwo \\
\cmidrule(lr){3-7}
5 & 2 & 5034.7$\times$ & 4271.9$\times$ & 140.6$\times$ & 358.8$\times$ & 2.1$\times$ \\
5 & 5 & 8464.9$\times$ & 10375.2$\times$ & 241.0$\times$ & 1513.0$\times$ & 1.8$\times$ \\
10 & 2 & 10750.1$\times$ & 27447.3$\times$ & 187.2$\times$ & 2685.5$\times$ & 2.2$\times$ \\
10 & 5 & 5767.1$\times$ & 4850.8$\times$ & 127.8$\times$ & 383.3$\times$ & 2.3$\times$ \\
\bottomrule
\end{tabular}%
}
\end{table}


\begin{table}[htbp]
\centering
\caption{Median of iterations and wall-clock time (s) -- Example \ref{Example_1}.B \& C}
\label{tab:results_Example_1_B_C}
{%
\begin{tabular}{rrcccccccc}
\toprule
$n$ & $m$ & \multicolumn{2}{c}{\AlgBIOne} & \multicolumn{2}{c}{\AlgOne} & \multicolumn{2}{c}{\AlgBITwo} & \multicolumn{2}{c}{\AlgTwo} \\
 &  & Iter & Time (s) & Iter & Time (s) & Iter & Time (s) & Iter & Time (s) \\
\cmidrule(lr){3-4} \cmidrule(lr){5-6} \cmidrule(lr){7-8} \cmidrule(lr){9-10} 
50 & 5 & 3851 & \num{2.01e-02} & 16 & \textbf{\num{7.21e-05}} & 9103 & \num{1.57e-01} & \textbf{10} & \num{1.72e-04} \\
50 & 8 & 4906 & \num{3.70e-02} & 21 & \textbf{\num{1.30e-04}} & 9396 & \num{2.77e-01} & \textbf{11} & \num{3.02e-04} \\
100 & 5 & 3976 & \num{6.40e-02} & 28 & \textbf{\num{2.97e-04}} & 9340 & \num{5.88e-01} & \textbf{15} & \num{6.90e-04} \\
100 & 8 & 4510 & \num{1.03e-01} & 22 & \textbf{\num{3.36e-04}} & 8998 & \num{7.53e-01} & \textbf{12} & \num{8.48e-04} \\
100 & 20 & 32712 & \num{1.60e+00} & 17 & \textbf{\num{6.04e-04}} & 1207 & \num{2.49e-01} & \textbf{5} & \num{1.92e-03} \\
100 & 30 & 38040 & \num{2.66e+00} & 19 & \textbf{\num{9.66e-04}} & 1218 & \num{4.51e-01} & \textbf{9} & \num{3.04e-03} \\
100 & 50 & 41405 & \num{4.68e+00} & 16 & \textbf{\num{1.20e-03}} & 1600 & \num{8.26e-01} & \textbf{7} & \num{5.10e-03} \\
200 & 20 & 31868 & \num{6.86e+00} & 24 & \textbf{\num{2.41e-03}} & 1323 & \num{1.28e+00} & \textbf{13} & \num{1.03e-02} \\
200 & 30 & 29778 & \num{9.80e+00} & 21 & \textbf{\num{2.80e-03}} & 1266 & \num{2.21e+00} & \textbf{11} & \num{1.43e-02} \\
200 & 50 & 23601 & \num{1.18e+01} & 20 & \textbf{\num{3.95e-03}} & 984 & \num{2.22e+00} & \textbf{6} & \num{2.18e-02} \\
500 & 20 & 17461 & \num{1.40e+01} & 22 & \textbf{\num{1.08e-02}} & 851 & \num{4.22e+00} & \textbf{12} & \num{4.43e-02} \\
500 & 30 & 14873 & \num{1.72e+01} & 17 & \textbf{\num{1.05e-02}} & 701 & \num{5.13e+00} & \textbf{7} & \num{5.71e-02} \\
500 & 50 & 978 & \num{2.51e+00} & 32 & \textbf{\num{6.15e-02}} & 549 & \num{5.68e+00} & \textbf{6} & \num{1.58e-01} \\
\bottomrule
\end{tabular}%
}
\end{table}

\begin{table}[htbp]
\centering
\caption{Speedup relative to \AlgOne -- Example \ref{Example_1}.B \& C}
\label{tab:speedup_Example_1_B_C}
{%
\begin{tabular}{rrccc}
\toprule
$n$ & $m$ & \AlgBIOne & \AlgBITwo & \AlgTwo \\
\cmidrule(lr){3-5}
50 & 5 & 278.3$\times$ & 2176.7$\times$ & 2.4$\times$ \\
50 & 8 & 285.7$\times$ & 2135.6$\times$ & 2.3$\times$ \\
100 & 5 & 215.2$\times$ & 1978.7$\times$ & 2.3$\times$ \\
100 & 8 & 305.2$\times$ & 2237.8$\times$ & 2.5$\times$ \\
100 & 20 & 2646.6$\times$ & 411.5$\times$ & 3.2$\times$ \\
100 & 30 & 2753.0$\times$ & 466.5$\times$ & 3.1$\times$ \\
100 & 50 & 3908.3$\times$ & 689.0$\times$ & 4.3$\times$ \\
200 & 20 & 2846.2$\times$ & 530.7$\times$ & 4.3$\times$ \\
200 & 30 & 3494.3$\times$ & 787.1$\times$ & 5.1$\times$ \\
200 & 50 & 2992.8$\times$ & 563.3$\times$ & 5.5$\times$ \\
500 & 20 & 1299.7$\times$ & 391.3$\times$ & 4.1$\times$ \\
500 & 30 & 1642.5$\times$ & 489.5$\times$ & 5.5$\times$ \\
500 & 50 & 40.7$\times$ & 92.3$\times$ & 2.6$\times$ \\
\bottomrule
\end{tabular}%
}
\end{table}


\begin{table}[htbp]
\centering
\caption{Median of iterations and wall-clock time (s) -- Example \ref{Example_2}.A}
\label{tab:results_Example_2_A}
\resizebox{\textwidth}{!}{%
\begin{tabular}{rrcccccccccccc}
\toprule
$n$ & $m$ & \multicolumn{2}{c}{\EGM} & \multicolumn{2}{c}{\AlgMalAdap} & \multicolumn{2}{c}{\AlgBIOne} & \multicolumn{2}{c}{\AlgOne} & \multicolumn{2}{c}{\AlgBITwo} & \multicolumn{2}{c}{\AlgTwo} \\
 &  & Iter & Time (s) & Iter & Time (s) & Iter & Time (s) & Iter & Time (s) & Iter & Time (s) & Iter & Time (s) \\
\cmidrule(lr){3-4} \cmidrule(lr){5-6} \cmidrule(lr){7-8} \cmidrule(lr){9-10} \cmidrule(lr){11-12} \cmidrule(lr){13-14} 
5 & 2 & 45 & \num{7.26e-02} & 38 & \num{1.26e-01} & 1154 & \num{1.04e-03} & 19 & \textbf{\num{1.45e-05}} & 1759 & \num{3.20e-03} & \textbf{5} & \num{3.09e-05} \\
5 & 5 & 48 & \num{1.87e-01} & 62 & \num{3.16e-01} & 2566 & \num{3.22e-03} & 15 & \textbf{\num{2.27e-05}} & 10158 & \num{9.56e-02} & \textbf{6} & \num{4.54e-05} \\
10 & 2 & 34 & \num{1.75e-01} & 60 & \num{1.54e-01} & 4112 & \num{4.47e-03} & 18 & \textbf{\num{1.85e-05}} & 6378 & \num{1.39e-02} & \textbf{5} & \num{3.17e-05} \\
10 & 5 & 40 & \num{2.04e-01} & 67 & \num{3.73e-01} & 2440 & \num{4.93e-03} & 14 & \textbf{\num{3.06e-05}} & 10108 & \num{7.54e-02} & \textbf{6} & \num{5.74e-05} \\
\bottomrule
\end{tabular}%
}
\end{table}

\begin{table}[htbp]
\centering
\caption{Speedup relative to \AlgOne -- Example \ref{Example_2}.A}
\label{tab:speedup_Example_2_A}
{%
\begin{tabular}{rrccccc}
\toprule
$n$ & $m$ & \EGM & \AlgMalAdap & \AlgBIOne & \AlgBITwo & \AlgTwo \\
\cmidrule(lr){3-7}
5 & 2 & 5018.5$\times$ & 8708.8$\times$ & 72.1$\times$ & 221.5$\times$ & 2.1$\times$ \\
5 & 5 & 8226.5$\times$ & 13931.2$\times$ & 142.1$\times$ & 4215.2$\times$ & 2.0$\times$ \\
10 & 2 & 9472.0$\times$ & 8334.1$\times$ & 242.1$\times$ & 751.6$\times$ & 1.7$\times$ \\
10 & 5 & 6670.9$\times$ & 12186.8$\times$ & 161.1$\times$ & 2463.0$\times$ & 1.9$\times$ \\
\bottomrule
\end{tabular}%
}
\end{table}


\begin{table}[htbp]
\centering
\caption{Median of iterations and wall-clock time (s) -- Example \ref{Example_2}.B \& C}
\label{tab:results_Example_2_B_C}
{%
\begin{tabular}{rrcccccccc}
\toprule
$n$ & $m$ & \multicolumn{2}{c}{\AlgBIOne} & \multicolumn{2}{c}{\AlgOne} & \multicolumn{2}{c}{\AlgBITwo} & \multicolumn{2}{c}{\AlgTwo} \\
 &  & Iter & Time (s) & Iter & Time (s) & Iter & Time (s) & Iter & Time (s) \\
\cmidrule(lr){3-4} \cmidrule(lr){5-6} \cmidrule(lr){7-8} \cmidrule(lr){9-10} 
50 & 5 & 5131 & \num{3.90e-02} & 19 & \textbf{\num{1.17e-04}} & 11904 & \num{3.01e-01} & \textbf{11} & \num{2.38e-04} \\
50 & 8 & 4971 & \num{5.07e-02} & 26 & \textbf{\num{2.21e-04}} & 9198 & \num{3.84e-01} & \textbf{12} & \num{4.30e-04} \\
100 & 5 & 3738 & \num{7.92e-02} & 18 & \textbf{\num{2.41e-04}} & 8226 & \num{6.05e-01} & \textbf{8} & \num{5.63e-04} \\
100 & 8 & 3331 & \num{8.52e-02} & 13 & \textbf{\num{2.33e-04}} & 9212 & \num{8.77e-01} & \textbf{4} & \num{4.87e-04} \\
100 & 20 & 2266 & \num{1.34e-01} & 16 & \textbf{\num{8.64e-04}} & 1108 & \num{3.15e-01} & \textbf{7} & \num{3.89e-03} \\
100 & 30 & 38772 & \num{3.30e+00} & 17 & \textbf{\num{1.18e-03}} & 1209 & \num{4.87e-01} & \textbf{7} & \num{5.05e-03} \\
100 & 50 & 1054 & \num{1.49e-01} & 15 & \textbf{\num{1.49e-03}} & 1291 & \num{5.43e-01} & \textbf{6} & \num{7.00e-03} \\
200 & 20 & 24715 & \num{6.41e+00} & 16 & \textbf{\num{3.76e-03}} & 1075 & \num{1.24e+00} & \textbf{7} & \num{1.19e-02} \\
200 & 30 & 25707 & \num{9.24e+00} & 21 & \textbf{\num{5.72e-03}} & 1113 & \num{2.41e+00} & \textbf{8} & \num{2.57e-02} \\
200 & 50 & 822 & \num{4.05e-01} & 22 & \textbf{\num{7.90e-03}} & 867 & \num{1.44e+00} & \textbf{5} & \num{3.76e-02} \\
500 & 20 & 18475 & \num{2.30e+01} & 19 & \textbf{\num{1.58e-02}} & 813 & \num{6.68e+00} & \textbf{9} & \num{6.81e-02} \\
500 & 30 & 13732 & \num{2.58e+01} & 12 & \textbf{\num{1.56e-02}} & 738 & \num{7.18e+00} & \textbf{3} & \num{5.19e-02} \\
500 & 50 & 16452 & \num{5.33e+01} & 33 & \textbf{\num{6.73e-02}} & 799 & \num{1.45e+01} & \textbf{6} & \num{3.29e-01} \\
\bottomrule
\end{tabular}%
}
\end{table}

\begin{table}[htbp]
\centering
\caption{Speedup relative to \AlgOne -- Example \ref{Example_2}.B \& C}
\label{tab:speedup_Example_2_B_C}
{%
\begin{tabular}{rrccc}
\toprule
$n$ & $m$ & \AlgBIOne & \AlgBITwo & \AlgTwo \\
\cmidrule(lr){3-5}
50 & 5 & 332.1$\times$ & 2562.7$\times$ & 2.0$\times$ \\
50 & 8 & 229.3$\times$ & 1734.2$\times$ & 1.9$\times$ \\
100 & 5 & 328.6$\times$ & 2508.0$\times$ & 2.3$\times$ \\
100 & 8 & 366.0$\times$ & 3764.6$\times$ & 2.1$\times$ \\
100 & 20 & 154.8$\times$ & 364.3$\times$ & 4.5$\times$ \\
100 & 30 & 2791.2$\times$ & 411.5$\times$ & 4.3$\times$ \\
100 & 50 & 100.1$\times$ & 364.6$\times$ & 4.7$\times$ \\
200 & 20 & 1704.4$\times$ & 329.4$\times$ & 3.2$\times$ \\
200 & 30 & 1614.1$\times$ & 421.9$\times$ & 4.5$\times$ \\
200 & 50 & 51.3$\times$ & 182.5$\times$ & 4.8$\times$ \\
500 & 20 & 1450.8$\times$ & 422.0$\times$ & 4.3$\times$ \\
500 & 30 & 1650.3$\times$ & 459.7$\times$ & 3.3$\times$ \\
500 & 50 & 791.4$\times$ & 215.9$\times$ & 4.9$\times$ \\
\bottomrule
\end{tabular}%
}
\end{table}


\begin{table}[htbp]
\centering
\caption{Median of iterations and wall-clock time (s) -- Example \ref{Example_3}.A}
\label{tab:results_Example_3_A}
\resizebox{\textwidth}{!}{%
\begin{tabular}{rrcccccccccccc}
\toprule
$n$ & $m$ & \multicolumn{2}{c}{\EGM} & \multicolumn{2}{c}{\AlgMalAdap} & \multicolumn{2}{c}{\AlgBIOne} & \multicolumn{2}{c}{\AlgOne} & \multicolumn{2}{c}{\AlgBITwo} & \multicolumn{2}{c}{\AlgTwo} \\
 &  & Iter & Time (s) & Iter & Time (s) & Iter & Time (s) & Iter & Time (s) & Iter & Time (s) & Iter & Time (s) \\
\cmidrule(lr){3-4} \cmidrule(lr){5-6} \cmidrule(lr){7-8} \cmidrule(lr){9-10} \cmidrule(lr){11-12} \cmidrule(lr){13-14} 
5 & 2 & 58 & \num{1.93e-01} & 34 & \num{2.51e-01} & 2569 & \num{3.02e-03} & 17 & \textbf{\num{1.89e-05}} & 9549 & \num{4.52e-02} & \textbf{11} & \num{3.57e-05} \\
5 & 5 & 53 & \num{1.41e-01} & 45 & \num{2.92e-01} & 4256 & \num{7.02e-03} & 11 & \textbf{\num{2.10e-05}} & 14770 & \num{8.83e-02} & \textbf{6} & \num{5.18e-05} \\
10 & 2 & 29 & \num{2.60e-01} & 56 & \num{5.24e-01} & 6978 & \num{9.27e-03} & 12 & \textbf{\num{1.46e-05}} & 15400 & \num{8.00e-02} & \textbf{4} & \num{2.94e-05} \\
10 & 5 & 26 & \num{3.33e-01} & 47 & \num{6.35e-01} & 4815 & \num{9.49e-03} & 12 & \textbf{\num{2.70e-05}} & 12973 & \num{1.06e-01} & \textbf{4} & \num{5.14e-05} \\
\bottomrule
\end{tabular}%
}
\end{table}

\begin{table}[htbp]
\centering
\caption{Speedup relative to \AlgOne -- Example \ref{Example_3}.A}
\label{tab:speedup_Example_3_A}
{%
\begin{tabular}{rrccccc}
\toprule
$n$ & $m$ & \EGM & \AlgMalAdap & \AlgBIOne & \AlgBITwo & \AlgTwo \\
\cmidrule(lr){3-7}
5 & 2 & 10176.0$\times$ & 13250.2$\times$ & 159.3$\times$ & 2387.2$\times$ & 1.9$\times$ \\
5 & 5 & 6710.9$\times$ & 13859.2$\times$ & 333.6$\times$ & 4196.4$\times$ & 2.5$\times$ \\
10 & 2 & 17740.5$\times$ & 35812.9$\times$ & 632.9$\times$ & 5462.6$\times$ & 2.0$\times$ \\
10 & 5 & 12312.3$\times$ & 23471.9$\times$ & 351.0$\times$ & 3938.4$\times$ & 1.9$\times$ \\
\bottomrule
\end{tabular}%
}
\end{table}


\begin{table}[htbp]
\centering
\caption{Median of iterations and wall-clock time (s) -- Example \ref{Example_3}.B \& C}
\label{tab:results_Example_3_B_C}
{%
\begin{tabular}{rrcccccccc}
\toprule
$n$ & $m$ & \multicolumn{2}{c}{\AlgBIOne} & \multicolumn{2}{c}{\AlgOne} & \multicolumn{2}{c}{\AlgBITwo} & \multicolumn{2}{c}{\AlgTwo} \\
 &  & Iter & Time (s) & Iter & Time (s) & Iter & Time (s) & Iter & Time (s) \\
\cmidrule(lr){3-4} \cmidrule(lr){5-6} \cmidrule(lr){7-8} \cmidrule(lr){9-10} 
50 & 5 & 5647 & \num{4.29e-02} & 16 & \textbf{\num{9.55e-05}} & 13759 & \num{4.13e-01} & \textbf{10} & \num{2.13e-04} \\
50 & 8 & 6027 & \num{6.18e-02} & 20 & \textbf{\num{1.69e-04}} & 13053 & \num{6.40e-01} & \textbf{8} & \num{3.45e-04} \\
100 & 5 & -- & -- & 10 & \textbf{\num{1.34e-04}} & 10839 & \num{1.14e+00} & \textbf{4} & \num{2.52e-04} \\
100 & 8 & 5298 & \num{1.68e-01} & 20 & \textbf{\num{3.61e-04}} & 10900 & \num{1.36e+00} & \textbf{9} & \num{8.28e-04} \\
100 & 20 & 32921 & \num{2.02e+00} & 16 & \textbf{\num{8.94e-04}} & 1346 & \num{3.84e-01} & \textbf{8} & \num{3.57e-03} \\
100 & 30 & 38636 & \num{3.60e+00} & 18 & \textbf{\num{1.27e-03}} & 1573 & \num{6.56e-01} & \textbf{6} & \num{4.95e-03} \\
100 & 50 & 40032 & \num{4.64e+00} & 16 & \textbf{\num{9.25e-04}} & 1678 & \num{1.07e+00} & \textbf{8} & \num{4.56e-03} \\
200 & 20 & 33695 & \num{7.84e+00} & 22 & \textbf{\num{2.72e-03}} & 1394 & \num{2.08e+00} & \textbf{9} & \num{1.12e-02} \\
200 & 30 & 32940 & \num{1.07e+01} & 15 & \textbf{\num{2.23e-03}} & 1228 & \num{3.21e+00} & \textbf{7} & \num{1.18e-02} \\
200 & 50 & 23959 & \num{1.25e+01} & 15 & \textbf{\num{2.95e-03}} & 1198 & \num{3.13e+00} & \textbf{4} & \num{1.24e-02} \\
500 & 20 & 18383 & \num{1.52e+01} & 23 & \textbf{\num{1.60e-02}} & 864 & \num{6.97e+00} & \textbf{7} & \num{6.30e-02} \\
500 & 30 & 19819 & \num{2.78e+01} & 17 & \textbf{\num{1.79e-02}} & 947 & \num{1.13e+01} & \textbf{7} & \num{6.88e-02} \\
500 & 50 & 15070 & \num{3.92e+01} & 9 & \textbf{\num{1.67e-02}} & 876 & \num{1.74e+01} & \textbf{3} & \num{5.15e-02} \\
\bottomrule
\end{tabular}%
}
\end{table}

\begin{table}[htbp]
\centering
\caption{Speedup relative to \AlgOne -- Example \ref{Example_3}.B \& C}
\label{tab:speedup_Example_3_B_C}
{%
\begin{tabular}{rrccc}
\toprule
$n$ & $m$ & \AlgBIOne & \AlgBITwo & \AlgTwo \\
\cmidrule(lr){3-5}
50 & 5 & 448.9$\times$ & 4325.5$\times$ & 2.2$\times$ \\
50 & 8 & 366.5$\times$ & 3800.6$\times$ & 2.0$\times$ \\
100 & 5 & -- & 8489.1$\times$ & 1.9$\times$ \\
100 & 8 & 465.9$\times$ & 3774.2$\times$ & 2.3$\times$ \\
100 & 20 & 2262.7$\times$ & 429.7$\times$ & 4.0$\times$ \\
100 & 30 & 2844.0$\times$ & 517.3$\times$ & 3.9$\times$ \\
100 & 50 & 5021.2$\times$ & 1158.0$\times$ & 4.9$\times$ \\
200 & 20 & 2881.7$\times$ & 764.8$\times$ & 4.1$\times$ \\
200 & 30 & 4805.9$\times$ & 1439.7$\times$ & 5.3$\times$ \\
200 & 50 & 4246.1$\times$ & 1060.3$\times$ & 4.2$\times$ \\
500 & 20 & 955.0$\times$ & 436.5$\times$ & 3.9$\times$ \\
500 & 30 & 1552.9$\times$ & 634.4$\times$ & 3.8$\times$ \\
500 & 50 & 2352.7$\times$ & 1044.2$\times$ & 3.1$\times$ \\
\bottomrule
\end{tabular}%
}
\end{table}

\end{document}